\documentclass[preprint]{elsarticlecomm}

\usepackage{hyperref}
\usepackage{amsfonts,amsmath,amssymb}
\usepackage{latexsym}

\newcommand{\CC}{{\mathbb C}}
\newcommand{\RR}{{\mathbb R}}
\newcommand{\ZZ}{{\mathbb Z}}
\newcommand{\NN}{{\mathbb N}}
\newcommand{\TT}{{\mathbb T}}
\newcommand{\DD}{{\mathbb D}}
\newcommand{\II}{{\mathbb I}}
\newcommand{\HP}{{\mathbb H}}
\newcommand{\HH}{{\mathcal H}}
\newcommand{\Z}{{\mathcal Z}}
\newcommand{\W}{{\mathcal W}}

\newcommand{\PP}{{\mathcal P}}

\newcommand{\PK}{{\mathcal P}_n(K)}
\newcommand{\DK}{\partial K}

\newcommand{\al}{\alpha}
\newcommand{\as}{\alpha^{*}}
\newcommand{\aw}{\widetilde{\alpha}}
\newcommand{\aws}{\widetilde{\alpha^{*}}}
\newcommand{\gs}{\gamma^{*}}
\newcommand{\gw}{\widetilde{\gamma}}
\newcommand{\gws}{\widetilde{\gamma^{*}}}
\newcommand{\ze}{\zeta}

\newcommand{\bnu}{\boldsymbol{\nu}}
\newcommand{\bmu}{\boldsymbol{\mu}}
\newcommand{\bnun}{{\boldsymbol{\nu}}_n}
\newcommand{\bmun}{\boldsymbol{\mu}_n}

\newcommand{\Norm}[1]{{\left\|{#1}\right\|}}
\newcommand{\Inner}[2]{{\left\langle #1, #2 \right\rangle}}

\newcommand{\capa}{\mathop{\mathrm{cap}}}
\newcommand{\intt}{\mathop{\mathrm{int}}}
\newcommand{\diam}{\mathop{\mathrm{diam}}}
\newcommand{\width}{\mathop{\mathrm{width}}}
\newcommand{\dist}{\mathop{\mathrm{dist}}}

\newcounter{othm}
\def\theothm{\Alph{othm}}
\newenvironment{othm}{
  \em
  \vskip 0.10in
  \refstepcounter{othm}
  \noindent{\bf Theorem\ \theothm}
}{\vskip 0.10in}

\newenvironment{olemma}{
  \em
  \vskip 0.10in
  \refstepcounter{othm}
  \noindent{\bf Lemma\ \theothm}
}{\vskip 0.10in}

\newtheorem{theorem}{Theorem}
\newtheorem{corollary}{Corollary}
\newtheorem{lemma}{Lemma}
\newtheorem{proposition}{Proposition}

\newtheorem{definition}{Definition}
\newtheorem{conjecture}{Conjecture}

 \newdefinition{remark}{Remark}

 \newproof{proof}{Proof}

\begin{document}

\begin{frontmatter}

\title{Tur\'an type oscillation inequalities in $L^q$ norm on the boundary of convex domains}

\author[label3,label4]{Polina Yu. Glazyrina}
 \ead{polina.glazyrina@usu.ru}

 \address[label3]{Institute of Mathematics and Computer Sciences, Ural Federal University,
 Ekaterinburg, Russia}

\address[label4]{Institute of Mathematics and Mechanics, Ural Branch of the Russian Academy of Sciences,
 Ekaterinburg, Russia}

\author[label1,label2]{Szil\'ard Gy. R\'ev\'esz\corref{cor1}}
\ead{revesz.szilard@renyi.mta.hu}
  \address[label1]{Institute of Mathematics, Faculty of Sciences,  Budapest University of Technology and Economics,
 Budapest, M\H uegyetem rkp. 3-9}
 \address[label2]{A. R\'enyi Institute of Mathematics Hungarian
Academy of Sciences,  Budapest, P.O.B. 127, 1364 Hungary}

\cortext[cor1]{Corresponding author}

\begin{abstract}
Some 76 years ago Paul Tur\'an was the first to establish lower estimations of the ratio of the maximum norm of the derivatives of polynomials and the maximum norm of the polynomials themselves on the interval $\II:=[-1,1]$ and on the unit disk $\DD:=\{z\in\CC~:~|z|\le 1\}$ under the normalization condition that the zeroes of the polynomial all lie in the interval or in the disk, respectively. He proved that with $n:=\deg p$ tending to infinity, the precise growth order of the minimal possible ratio of the derivative norm and the norm is $\sqrt{n}$ for $\II$ and $n$ for $\DD$.

Er\H od continued the work of Tur\'an and extended his results to several other domains. The growth of the minimal possible ratio of the $\infty$-norm of the derivative and the polynomial itself was proved to be of order $n$ for all compact convex domains a decade ago.

Although Tur\'an himself gave comments about the above
oscillation question in $L^q$ norms, till recently results were
known only for $\DD$ and $\II$. Here we prove that in $L^q$
norm the oscillation order is again $n$ for a certain class of
convex domains, including all smooth convex domains and also
convex polygonal domains having no acute angles at their
vertices.
\end{abstract}


\end{frontmatter}


\section{Introduction}\label{sec:intro}

\subsection{The oscillation of a polynomial in maximum norm}\label{sec:maxnormopsc}

At the turn of the 19th and 20th centuries, the first estimates of the derivative of a polynomial via the maximum of its
values  appeared.
They were obtain by V.Markov in 1889, for  algebraic polynomials on an interval,
 by Bernstein and M.~Riesz in 1914, for  trigonometric polynomials
on $[0,2\pi]$ and algebraic polynomials on the unit circle.
In 1923,  Szeg\H o    \cite{Szego23} obtained an estimate for a large class of (not necessarily convex, but piecewise smooth)
 domains. Namely, if $K\subset \CC$ is a piecewise smooth simply connected domain, with its boundary consisting of finitely many analytic
Jordan arcs, and if the maximum of the outer angles at the joining
 vertices of these arcs is\footnote{If the domain is bounded, then for all directions it has supporting lines,
whence there are points where the outer angle is at least $\pi$.} $\beta \in [\pi,2\pi]$, then the domain admits
a Markov type inequality of the form $\Norm{p'}_K\le c_K n^{\beta/2\pi} \Norm{p}_K$
for any polynomial $p$ of degree $n$.
Here the norm $\Norm{\cdot}:=\Norm{\cdot}_K$ denotes sup norm over values attained on $K$.
 This inequality is essentially sharp for all such domains. In particular, this immediately implies that for \emph{analytically smooth}
convex domains the  Markov factor is $O(n)$.
For the unit disk
$$\DD:=\{z\in \CC~:~ |z|\le 1\}
$$
even the exact inequality is well-known:
\begin{equation}\label{BR}
\Norm{p'}_\DD\le n \Norm{p}_\DD.
\end{equation}
This was conjectured, and almost proved, by Bernstein \cite{Bernstein, Bernstein_com}; for the first published proof see \cite{Riesz}.
Similarly, the precise result is also classical for the unit interval
$$\II:=[-1,1]$$ then we have Markov's Inequality $\Norm{p'}_\II\le n^2 \Norm{p}_\II$,
which is sharp\footnote{Note that in this case the outer angles at the break-points of the piecewise smooth boundary are exactly $2\pi$ at each end.}, see
\cite{Markov}.

In 1939 Paul Tur\'an started to study converse inequalities of
the form $$\Norm{p'}_K\ge c_K n^A \Norm{p}_K.$$ Clearly such a
converse can only hold if further restrictions are imposed on
the occurring polynomials $p$. Tur\'an assumed that all zeroes
of the polynomials belong to $K$. So denote the set of complex
(algebraic) polynomials of degree (exactly) $n$ as $\PP_n$, and
the subset with all the $n$ (complex) roots in some set
$K\subset\CC$ by $\PK$. Denote by $\Gamma$ the boundary of
$K$. The (normalized) quantity under our study
 is the ``inverse Markov
factor" or ``oscillation factor"
\begin{equation}\label{Mdef}
M_{n,q}(K):=\inf_{p\in \PK} M_q(p) \qquad \mbox{with} \qquad
M_q(p):=\frac{\Norm{p'}_{L^q(\Gamma)}}{\Norm{p}_{L^q(\Gamma)}},
\end{equation}
where, as usual,
\begin{align*}
\Norm{p}_{q}:&=\Norm{p}_{L^q(\Gamma)}:=\left(\int_{\Gamma} |p(z)|^q|dz|\right)^{1/q}, \quad (0<q<\infty) \notag
\\ \Norm{p}_K:&=\Norm{p}_{L^\infty(\Gamma)}=\sup_{z\in \Gamma}|p(z)|=\sup_{z\in K}|p(z)|.
\end{align*}

\begin{othm}{\bf(Tur\'an, \cite[p. 90]{Tur}).}\label{oth:Turandisk} If $p\in \PP_n(\DD)$,  then we have
\begin{equation}\label{Turandisk}
\Norm{p'}_\DD\ge \frac n2 \Norm{p}_\DD~.
\end{equation}
\end{othm}

\begin{othm}{\bf(Tur\'an, \cite[p. 91]{Tur}).}\label{oth:Turanint} If $p\in\PP_n(\II)$, then we have
\begin{equation*}
\Norm{p'}_\II\ge \frac {\sqrt{n}}{6} \Norm{p}_\II~.
\end{equation*}
\end{othm}

Inequality \eqref{Turandisk} of Theorem \ref{oth:Turandisk} is best possible.
Regarding Theorem \ref{oth:Turanint}, Tur\'an pointed out by
example of $(1-x^2)^{n}$ that the $\sqrt{n}$ order is sharp.
Some slightly improved constants can be found in
\cite{BabenkoMN86} and \cite{LP}, however, the exact value of
the constants and the corresponding extremal polynomials were
already computed for all fixed $n$ by Er\H{o}d in \cite{Er}.

Now we are going to describe results concerning Tur\'an-type
inequalities~\eqref{Mdef} for general convex sets. To study
\eqref{Mdef} some geometric parameters of the convex domain $K$
are involved naturally. We write $d:=d(K):=\diam (K)$ for the
{\em diameter} of $K$, and $w:=w(K):={\width}(K)$ for the {\em
minimal width} of $K$. That is,
\begin{equation*}
d(K):= \max_{z', z''\in K} |z'-z''|,
\end{equation*}
\begin{equation*}
w(K):= \min_{\gamma\in [-\pi,\pi]} \left( \max_{z\in K} \Re
(ze^{i\gamma}) - \min_{z\in K} \Re (ze^{i\gamma}) \right).
\end{equation*}
Note that a (closed) convex domain is a (closed), bounded, convex set
$K\subset\CC$ with nonempty interior, hence $0<w(K)\le
d(K)<\infty$.

The key to Theorem \ref{oth:Turandisk} was the following
observation, which had already been present implicitly in
\cite[the footnote on p.~93]{Tur} and \cite{Er} and was later
formulated explicitly  by Levenberg and Poletsky  in \cite[Proposition 2.1]{LP}.
\begin{olemma}{\bf(Tur\'an).}\label{Tlemma} Assume that $z\in\partial K$ and
that there exists a disc $D_R=\{\zeta\in \CC~:~ |\ze-z_0|=R\}$ of radius $R$ so that $z\in\partial D_R$ and $K\subset D_R$. Then for all $p\in\PK$ we have
\begin{equation}\label{Rdisc}
|p'(z)| \ge \frac n{2R} |p(z)|.
\end{equation}
\end{olemma}
The proof of this is really easy, so let us recall it for
completeness.
\begin{proof} As $~\dfrac{p'}{p}(z)=\sum_j \dfrac{1}{z-z_j}~$ and $~\Re
\dfrac{1}{1-\zeta} \geq 1/2 \quad (\forall |\zeta|<1$), if all
$z_j\in D_R$, then
\begin{align*}
R\cdot\left|\frac{p'}{p}(z)\right| & \geq \Re \left\{ (z-z_0)\sum_j \frac{1}{z-z_j} \right\}
\\ &= \sum_j \Re \frac{{z-z_0}}{(z-z_0)-(z_j-z_0)} =  \sum_j \Re
\frac{1}{1-\frac{z_j-z_0}{z-z_0}} \geq \frac{n}{2}.
\end{align*}
\end{proof}

Given this elementary observation, Levenberg and Poletsky found
it worthwhile to formally define the crucial property of convex
sets, necessary for drawing such an easy and direct conclusion
on the inverse Markov factors.

\begin{definition}
A set $K\Subset \CC$ is called \emph{$R$-circular}, if for any  $z\in\partial K$ there exists a disk $D_R$ of radius $R$, such that
$z\in\partial D_R$ and $D_R\supset K$ .
\end{definition}

Thus for any $R$-circular $K$ and $p\in \PP_n(K)$ at the boundary point $z\in\partial K$
with $\|p\|_K=|p(z)|$ we can draw the disk $D_R$ and it follows
\begin{othm}{\bf (Er\H od; Levenberg-Poletsky).}\label{th:circular}
For an $R$-circular $K$ we have
\begin{equation*}\label{circularineq}
\Norm{p'}_K\ge \frac n{2R}  \Norm{p}_K \quad \left(\forall p\in\PK \right)\qquad
\text{\rm that is} \qquad M_n(K)\geq \frac n{2R} ~.
\end{equation*}
\end{othm}

There are many important examples of $R$-circular compact sets and domains. E.g. a union of two circular arcs, joining at a vertex of angle less than $\pi$, is always $R$-circular with some $R$. Smooth convex closed curves, together with the encircled convex domain $K$, with curvature exceeding $\kappa>0$ are always $R$-circular with $R=1/\kappa$ according to a classical theorem of Blaschke \cite{Bla}. Further extensions of the Blaschke Rolling Ball Theorem allows to  realize $R$-circularity of much more general convex curves $\gamma$.
\begin{olemma}{\bf(Strantzen).}\label{l:roughcurvature} Let the convex domain $K$ have boundary curve $\Gamma=\partial K$ and let $\kappa>0$ be a fixed constant. Assume that the convex boundary curve $\Gamma$ (which is twice differentiable linearly almost everywhere) satisfies the curvature condition $\ddot{\Gamma}\geq \kappa$ almost everywhere. Then to each
boundary point $\zeta\in\partial K$ there exists a disk $D_R$ of
radius $R=1/\kappa$, such that $\zeta\in\partial D_R$, and
$K\subset D_R$. That is, $K$ is $R=1/\kappa$-circular.
\end{olemma}
\begin{proof} This result is essentially the far-reaching, relatively
recent generalization of Blaschke's Rolling Ball Theorem by
Strantzen. A reference for it is Lemma 9.11 on p. 83 of
\cite{BS}. For more details on this, as well as for some new
approaches to the proof of this generalization of the classical
Blaschke Rolling Ball Theorem, see \cite{Rev4}.
\end{proof}

Obviously, the above entails an order $n$ oscillation result for all convex domains with this a.e. condition on the curvature of the boundary curve. This leads us to the topic of Tur\'an type oscillation problems for more general sets and domains.

Drawing from the work of Tur\'an, Er\H od \cite[p. 74]{Er} already addressed the question: ``For what kind of domains does he method of Tur\'an apply?'' Clearly, by ''applies'' he meant that it provides order $n$ oscillation for the derivative. Moreover, he introduced new ideas into the investigation -- including the application of Chebyshev's Inequality \eqref{capacity} below -- so clearly he did not simply pursue the effect of Tur\'an's original methods, but was indeed after the right oscillation order of general domains. In particular, he showed

\begin{othm}{\bf(Er\H od, \cite[p. 73]{Er}).}\label{th:ellipse}
Let $0<b<1$ and let $E_b$ denote the ellipse domain with major axes $[-1,1]$ and minor axes $[-ib,ib]$. Then for all $p\in\PP_n(E_b)$ we have
\begin{equation*}
\Norm{p'} \ge \frac {b}{2}  n \Norm{p}.
\end{equation*}
\end{othm}

Moreover, he elaborated on the inverse Markov factors belonging to domains with some favorable geometric properties.The most general domains with $M(K)\gg n$, found by Er\H od, were described on p. 77 of \cite{Er}.

\begin{othm}{\bf(Er\H od).}\label{oth:transfquarter} Let $K$ be any convex domain bounded by finitely many Jordan arcs, joining at vertices with angles $<\pi$, with all the arcs being $C^2$-smooth and being either straight lines of length $<\Delta(K)$, where $\Delta(K)$ stands for the transfinite diameter of $K$, or having positive curvature bounded away from $0$ by a fixed constant $\kappa>0$. Then there is a constant $c(K)$, such that $M_n(K)\geq c(K) n$ for all $n\in\NN$.
\end{othm}

Note that this latter result of Er\H od incorporates regular $k$-gons $G_k$ for large enough $k$, but not the square $Q=G_4$, because the side length $h$ of a square is larger than the quarter of the transfinite diameter $\Delta$: actually, $\Delta(Q)\approx 0.59017\dots h$, while for the regular $k$-gon of side length $h$ we have
$$
\Delta(G_k)= \frac{\Gamma(1/k)}{\sqrt{\pi}2^{1+2/k}\Gamma(1/2+1/k)} h
$$
(see e.g. \cite[p. 135]{Rans}), so $\Delta(G_k) > h$ iff $k\geq 7$. This implies $M_n(G_k) \ge c_k n$ for $k\ge 7$.

In \cite{E}, Erd\'elyi proved order $n$ oscillation for the square\footnote{Erd\'elyi also proves similar results on rhombuses, under the further condition of some symmetry of the polynomials in consideration -- e.g. if the polynomials are real, or odd. Note also that his work on the topic preceded \cite{Rev2} and apparently was accomplished without being aware of details of \cite{Er}.} $Q=G_4$, too. A result of \cite{Rev1} also implied $M_n(G_k) \ge c_k n$ for $k\ge 4$, but still not for a triangle.

To deal with the flat case of straight line boundary arcs, Er\H
od involved another approach, cf. \cite[p. 76]{Er}, appearing
later to be essential for obtaining a general answer formulated
in Theorem \ref{th:convexdomain} below, and playing an
essential role in many further developments, including ours
here. Namely, Er\H od quoted Faber \cite{Faber} for the
fundamental result of Chebyshev on the monic polynomial of
minimal norm on an interval. Since this approach will be
extensively applied also in our work, we summarized basic facts
regarding this in the below Section \ref{sec:Cheby}.

For a few further examples, remarks and open problems regarding
inverse Markov factors for various classes of compact sets
which are not necessarily convex, see \cite{LP, Rev2, Rev3}.

A lower estimate of the inverse Markov factor for any compact
convex set (and of the same order as was known for the
interval) was obtained in full generality only in 2002, see
\cite[Theorem 3.2]{LP}.

\begin{othm}{\bf(Levenberg-Poletsky).}\label{th:generalroot} If
$K\subset \CC$ is a compact, convex set, $d:=\diam{K}$ and $p\in
\PK$, then we have
\begin{equation*}
\Norm{p'}_K\ge \frac {\sqrt{n}}{20\,\diam(K)}  \Norm{p}_K~.
\end{equation*}
\end{othm}
Clearly, assuming boundedness is natural, since all polynomials have
$\Norm{p_n}_K\!\!=\infty$ when the set $K$ is unbounded. Also, restricting ourselves to \emph{closed} bounded sets -- i.e., to compact sets -- does not change the $\sup $ norm of polynomials under study, as all polynomials are continuous.

Recall that the term {\em convex domain} stands for a compact, convex subset of $\CC$ having nonempty interior. That is, assuming that $K$ is a (bounded, closed) convex \emph{domain}, not just a compact convex \emph{set}, means that we exclude only the case of the interval, for which already Tur\'an clarified that the order of oscillation is precisely $\sqrt{n}$.

So in order to clarify the order of oscillation for all compact convex sets it remains to clarify the order of oscillation for compact convex domains. The solution of this general problem\footnote{Preceding this, an intermediate result of order $n^{2/3}$ oscillation for all compact convex domains has been worked out in \cite{Rev1} -- in view of the later developments, this has not been published in a journal.} has been published in 2006, see \cite{Rev2}.

\begin{othm}{\bf(Hal\'{a}sz--R\'ev\'esz).}\label{th:convexdomain}
Let $K\subset \CC$ be any bounded convex domain.
Then for all  $p\in
\PK$ we have
\begin{equation*}
\Norm{p'}_K\ge 0.0003 \frac{w(K)}{d^2(K)} n  \Norm{p}_K~.
\end{equation*}
\end{othm}
\begin{remark}\label{rem:sharp} This indeed provides the precise order, for an even larger order than $n$ cannot occur, not for any particular compact set. Namely, let $K\subset\CC$ be any compact set with diameter $d:=\diam(K)$. Then for all $n$ there exists a polynomial $p\in\PK$ of degree exactly $n$ satisfying
\begin{equation*}
\Norm{p'} \leq ~ C'(K)~ n~ {\Norm{p}} \qquad \text{\rm with}
\qquad C'(K):= 1/ \diam(K).
\end{equation*}
Indeed, considering a diameter $[z_0,w_0]$ and the polynomial $p(z)=(z-z_0)^d$, the respective norm is $\|p\|_\infty=d^n$ while the derivative norm becomes $\|p'\|_\infty=nd^{n-1}$, both attained at $w_0\in K$.
\end{remark}
So, this settles the question of the \emph{order}, but not the precise \emph{dependence on the geometry}. However, up to an \emph{absolute constant factor}, even the dependence on the geometrical features of the domain was also clarified in \cite{Rev2}.
\begin{othm}{\bf(R\'ev\'esz).}\label{th:sharp} Let $K\subset\CC$ be any
compact, connected set with diameter $d$ and minimal width $w$.
Then for all $n>n_0:=n_0(K):= 2 (d/16w)^2 \log (d/16w)$ there
exists a polynomial $p\in\PK$ of degree exactly $n$ satisfying
\begin{equation*}
\Norm{p'}_K \leq  C'(K) n \Norm{p}_K \qquad \text{\rm with}
\qquad C'(K):= 600 ~\frac{w(K)}{d^2(K)}~.
\end{equation*}
\end{othm}

So, interestingly, it turned out that among all convex compacta
only intervals can have an inverse Markov constant of order
$\sqrt{n}$, while domains with nonempty interior have oscillation order $n$.

One may ask, how useful, how general these results are? One fundamental area of potential applications is the theory of orthogonal polynomials. Badkov \cite{Badkov92, Badkov95} applied  Tur\'an's inequality~\eqref{Turandisk} for estimations of polynomials orthogonal on the circle with respect to a weight. It is well known that  polynomials orthogonal on a circle or on an interval (with respect to some weight there) have all their zeros on the interval or inside the circle, respectively. That is, if $P_n$ is the $n^{\rm th}$ orthogonal polynomail, then certainly we have $P_n \in \PK$ and the above oscillation results apply.

Analogous phenomenon takes place in the case of a  rectifiable
curve or, more generally, a compact set with a measure. The
precise statements  can be found in \cite[Satz III]{Fejer22},
\cite{Saff}, \cite[\S 10.2]{Davis}, \cite[Ch XVI, 16.2,
(6)]{SzegoOP}, \cite{Szego21}.

The respective upper estimations, i.e. Bernstein-Markov type
inequalities were extensively studied under analogous
constraints for the zeroes \cite{ErdelyiSzabados2002}. Since
the oscillation results of Tur\'an type are also formulated
under zero restrictions, it is of interest to compare the upper
and lower estimations of these derivative norm estimates.

The first relevant result were asked about by Erd\H os and solved by Lax \cite{Lax}: this states that if the zeroes of a polynomial are all
\emph{outside} the unit circle, then the classical Bernstein Inequality \eqref{BR} can be improved to
$\|p'\|_\DD \leq \dfrac{n}{2} \|p\|_\DD$. For further study of the topic of constrained Bernstein-Markov Inequalities we refer the reader to
\cite[Theorem, p. 58]{Malik69}, \cite{ErdelyiSzabados2002, ArestovAM} and the references therein.

This improvement by the factor $1/2$ reminds us the Tur\'an inequality: the common extremal polynomial is indeed the boundary case $p_n(z):=z^n-1$ (and rotations of thereof). Note that here the Tur\'an type restriction of $p\in \PK$ does not allow any improvement: the extremal $z^n$ provides an oscillation of exactly $n$ in regard of the Bernstein Inequality. Therefore, this very first result in constrained Bernstein-Markov Inequalities already prompts us to consider classes of polynomials with taking the norm on one set, while restricting the location of zeroes to another one. Concretely, the above Lax result talks about the class ${\mathcal P}_n (\CC \setminus\DD)$, under the (maximum) norm on $\DD$ (or on the boundary circle $\partial \DD$). Analogously, we can consider $\PK$ under the norm on another set $L$: the respective oscillation factors we may denote by
\begin{equation*}\label{MLdef}
M_{n,L}(K):=\inf_{P\in \PK} M_{\|\cdot\|_L}(P) \qquad \text{\rm with} \qquad
M_{\|\cdot\|_L}(P):=\frac{\Norm{P'}_{L}}{\Norm{P}_{L}},
\end{equation*}
with the norm $\|\cdot\|_L$ being the maximum norm\footnote{Or,
more generally, one may even consider
$M_{n,\|\cdot\|}(K):=\inf_{P\in \PK} M_{\|\cdot\|}(P)$ with
$M_{\|\cdot\|}(P):=\frac{\Norm{P'}}{\Norm{P}}$ being an
arbitrarily fixed norm $\|\cdot\|$.} taken on the set $L$.

In connection to the Tur\'an topic, this has also been
investigated at least when the set $L$ is a disk: $L=D_R=\{z \
: \ |z|\le R \}$.

Malik \cite{Malik69}  ($R< 1$) and Govil \cite{Govil73} ($R>1$)
showed that

\begin{othm}{\bf(Malik, Govil).} For any $R\ge 0$ we have
$$
M_{n,D_R}(K)=\begin{cases}
\frac{n}{1+R}, & R\le 1, \\
\frac{n}{1+R^n}, & R \ge 1.
\end{cases}
$$
\end{othm}

See also \cite{Akopyan00, AzizAhemad, PaulShahSingh} and the
references therein. However, apart from these rather special
choices of concentric disks, we have not found any result in
the literature for more general situations.

\subsection{Pointwise and integral mean estimates of oscillation}\label{sec:Lqosci}

There are many papers dealing with the $L^q$-versions of Tur\'an's inequality for the (unit) disk $\DD$, the (unit) interval $\II$, or for the period (one dimensional torus or
circle) $\TT:=\RR/2\pi\ZZ$ (here with the understanding that we consider real trigonometric polynomials, not complex polynomials). A nice review of the results obtained before 1994 is given in \cite[Ch. 6, 6.2.6, 6.3.1]{MMR}.

The story started by an obvious observation. Namely, already
Tur\'an himself mentioned in \cite{Tur} that on the perimeter
of the disk $\DD$ -- and, as is easily observed, the same way
under conditions of Theorem~\ref{oth:Turandisk} -- actually a
more general pointwise inequality holds \emph{at all points} of
$\partial \DD$. Namely, for $p\in \PK$ we have
\begin{equation}\label{Rdisc}
|p'(z)| \ge \frac n{2} |p(z)|, \quad |z|=1,
\end{equation}
and as a corollary, for any $q>0$,
\begin{equation*}
\left(\int_{|z|=1}|p'(z)|^q|dz|\right)^{1/q} \ge  \frac n{2}\left( \int_{|z|=1}|p(z)|^q|dz|\right)^{1/q}.
\end{equation*}
In other words, the Tur\'an result Theorem \ref{oth:Turandisk}
extends to all norms on the perimeter, including all norms
$L^q(\partial \DD)$, and we have for all polynomials $p\in\PK$
\begin{equation*}
 \Norm{p'}_{L^q(\partial
\DD)}\ge \frac n2 \Norm{p}_{L^q(\partial \DD)}, \qquad
M_{n,q}(\DD) \geq \frac{n}{2}~.
\end{equation*}
The same way, for $R$-circular domains the result of Theorem \ref{circularineq} extends as
\begin{equation*}
\Norm{p'}_{L^q(\partial K)}\ge \frac{n}{2R} \Norm{p}_{L^q(\partial K)}, \qquad M_{n,q}(K) \geq \frac{n}{2R}~.
\end{equation*}
However, calculating the respective quantities for the function
$p_n(z):=z^n-1$, it turns out that this inequality is not the
best (at least not for the most symmetric choice $p_n$), and
the determination of the value of the precise constant, as well
as identifying the extremal functions, remains to be done.

A related question is to compare the maximum norm of $p'$ to
the $L^q$ norm of $p$ itself. In this regard, the exact
constant is known from the work of Malik in \cite{Malik84}.
\begin{othm}{\bf(Malik).} For the unit disc $\DD$ and any $0<q<\infty$ we have
\begin{equation*}
\|p'\|_{\DD} \ge  \left( \frac{\Gamma(q/2+1)}{2\sqrt{\pi}\Gamma(q/2+1/2)}\right)^{1/q}
~ \frac{n}{2} ~ \|p\|_{L^q(\partial \DD)},
\end{equation*}
moreover, the inequality is the best possible, equality occurring precisely for all $az^n+b$ with $|a|=|b|$.
\end{othm}
Some other results comparing different norms were obtained by
Saff and Sheil-Small \cite{SaffSheil-Small}, Rubinstein
\cite{Rubinstein}, Aziz and Ahemad \cite{AzizAhemad}, Paul,
Shah and Singh \cite{PaulShahSingh}.

In the paper \cite[Cor. 11.1]{Badkov92}, another proof of the
pointwise Tur\'an inequality~\eqref{Rdisc} is given. The proof
is based on the properties of orthogonal polynomials and the
Christoffel function.

The classical inequalities of Bernstein and Markov are generalized for various differential operators, too, see \cite{Arestov}. In this context, also Tur\'an type converses have been already investigated: namely,
Akopyan~\cite{Akopyan00} studied Tur\'an-type inequalities in $L^2$-norm on the circle for some generalization of the operator of differentiation.

The estimation of the $L^q$ norm, or of any norm including e.g. any weighted $L^q$-norms, goes the same way if we have a pointwise estimation for all, (or for linearly almost all), boundary points. Therefore, as above, we can formulate e.g. the next result, see \cite{Rev3}.
\begin{othm}{\bf.}\label{th:Lqcircular} Assume that the boundary curve $\gamma:[0,L]\to\Gamma:=\partial K$ of the convex domain $K$ satisfies at (linearly) almost all points the condition that it has a curvature, not smaller than a given positive constant $\kappa$, i.e. $\ddot{\gamma} \geq \kappa(>0)$ a.e.

Then for any norm $\|\cdot\|$, defined for some class of functions (including the polynomials) on the boundary curve, we have $\| p'\| \geq \dfrac{n}{2\kappa} \|p\|$. In particular, $M_{n,q}(K) \geq \dfrac{n}{2\kappa}$.
\end{othm}

In case we discuss maximum norms, one can assume that $p(z)$ is maximal, and it suffices to obtain a lower estimation of $p'(z)$ only at such a special point -- for general norms, however, this is not sufficient. The above results work only for we have a pointwise inequality of the same strength \emph{everywhere}, or almost everywhere. The situation becomes considerably more difficult, when such a statement cannot be proved. E.g. if the domain in question is not strictly convex, i.e. if there is a line segment on the boundary, then the zeroes of the polynomial can be arranged so that even some zeroes of the derivative lie on the boundary, and at such points $p'(z)$ -- even $p'(z)/p(z)$ -- can vanish. As a result, at such points no fixed lower estimation can be guaranteed, and lacking a uniformly valid pointwise comparision of $p'$ and $p$, a direct conclusion cannot be drawn either.

This explains why the case of the interval $\II$ already proved to be much more complicated for the integral mean norms.

In a series of papers \cite{Zhou84, Zhou86, Zhou92, Zhou93, Zhou95}, Zhou proved the inequality
\begin{equation*}
\left( \int_{-1}^1 |p^{(k)}(x)|^p dx \right)^{1/p}\ge C^{(k)}_{p,q}(n)
\left( \int_{-1}^1 |p(x)|^q dx \right)^{1/q},
\end{equation*}
in the case of $k=1$ and $0<p\le q\le \infty,$ ${1-1/p+1/q\ge 0}$ with the constant
$C^{(1)}_{p,q}(n)=c_{p,q} \left(\sqrt{n}\right)^{1-1/p+1/q}$.

The best possible constants $C^{(k)}_{p,q}(n)$ were found by Babenko and Pichugov~\cite{BabenkoU86} for $p=q=\infty,$ $k=2$, by Bojanov \cite{Bojanov93} for $1\le p \le \infty,$ $q=\infty$, $1\le k\le n$, and by Varma~\cite{Varma88_83} in the case of $p=q=2,$ $k=1$.

Exact Tur\'an-type inequalities for trigonometric polynomials in different $L^q$-metrics on $\TT$ were proved in \cite{BabenkoU86, BabenkoMN86, Bojanov96, Tyrygin88, TyryginDAN88, KBL}.

Other inequalities on $\II$, $\DD$, the positive semiaxes, or on $\TT$ in various weighted $L^q$-metrics
 can be found in \cite{Varma79, UV96, WangZhou02, Zhou2002, KBL, YuWei}.

As said above, we also have a direct result for $R$-circular domains, and $R$-circularity could be ascertained by some conditions on the curvature. However, apart from these, for general domains, the situation was much less clear. Here is the only result, known for convex domains in general, and formulated in \cite{LP}.

\begin{othm}{\bf (Levenberg-Poletsky).}\label{othm:LPpointwise}
Asume that $K$ is a compact convex subset of the complex plane.
Let $z \in \partial K$ and $p\in \PK$. Then there exists
another point $\zeta \in K$, of distance $|z-\zeta| \le c_K /
\sqrt{n}$, such that $|p'(\zeta)| \geq c'_K \sqrt{n} |p(z)|$.
\end{othm}

Clearly, this is too weak for proving anything on $\|p'\|_{L^q(\partial K)}$, for in general $\zeta \not\in \partial K$, and in any case the full set of such $\zeta$ points can well be just a finite point set (some $c_k/\sqrt{n}$ net of the boundary).

To obtain something in the $L^q(\partial K)$ norm, we need to prove pointwise estimates for much more points, essentially for the "majority" of the points, with the respective $\zeta$ points strictly lying on the boundary $\partial K$, and in fact essentially we cannot allow $\zeta$ be different from $z$ (and so perhaps coincide for a large set of points $z$). We undertake this, also aiming at stronger inequalities, than the $\sqrt{n}$ order in the above result.

\section{Statement of new results}

In the later parts of our paper we will prove a rather general
main result, the formulation of which, however, requires some
preparations, i.e. certain geometrical notions and definitions,
to be developed first. Therefore, here we give only the two
main corollaries of the below Theorem \ref{th:posdepth}, which
provide us the main motivation for the whole study.

\begin{theorem}\label{th:smooth} Let $K\Subset \CC$ be any
smooth convex domain on the plane. Then there
exists a positive constant $C=C_K$, such that we have
$M_{n,q}(K) > C_K n$ for all $n\in \NN$.
\end{theorem}

Recall that we use the terminology of being \emph{smooth} in the sense that the
boundary curve $\gamma:\RR \to \partial K$ is differentiable --
i.e. it has a (unique) tangent at each points. In case of convex domains this also implies that $\gamma \in C^1(\RR)$, but still we did not assume $\gamma$ to be $C^2(\RR)$, as is quite usual in (classical) convex geometry.

That the very nature of this result does not really depend on
smoothness, is well seen from our next result.

\begin{theorem}\label{th:nonacutepolygon} Let $K\Subset \CC$ be
any convex (non-degenerate, i.e. bounded and with nonempty
interior) polygon on the plane, with no acute angles at its
vertices.
Then there exists a positive constant $C=C_K$, such
that we have $M_{n,q}(K) > C_K n$ for all $n\in \NN$.
\end{theorem}

\begin{remark} For the infinity norm case, clarifying the
situation (the order of growth of $M_n(G_k)$) for regular
$k$-gons took long. Here we see that for any $k\ge 4$ the
regular $k$-gon has $M_n(G_k)>C_k n$. However, the regular (and
any other) triangle, necessarily having some acute angles,
turns to be an entirely different case, the treatment of which
requires further ideas, too. We hope to return to that in a
subsequent paper.
\end{remark}

The organization of the material in the current work is as
follows. Next we summarize classical results on the Chebyshev
constant and transfinite diameter, and then we continue in the
subsequent section by introducing a few geometrical notions,
necessary to the formulation of our further results.

In Section \ref{sec:posdepth} we formulate and prove our main
result, together with a certain pointwise result of independent
interest. The below Theorem \ref{th:posdepth} directly implies both the above
stated Theorems \ref{th:smooth} and \ref{th:nonacutepolygon} as
corollaries (which is seen from Proposition \ref{p:anglesmooth} and Corollary \ref{cor:smooth}), whence we may indeed say that Theorem \ref{th:posdepth} is our main result in this paper.

Next, in Section \ref{sec:osciupper} we prove that in rather general situations--and surely for all convex domains--the investigated oscillation factor is \emph{at most} a constant times $n$, thus clarifying at least the order of growth (with the degree $n$) of this factor for the domains discussed in our results.

Finally, in the concluding section we formulate our conjecture about the general situation, and offer some further comments.

\section{Chebyshev estimates and transfinite diameter}\label{sec:Cheby}

In this section we summarize classical, yet powerful results
going back to Chebyshev.
\begin{olemma}{\bf (Chebyshev).}\label{l:capacity}
Let $J=[u,v]$ be any interval on the complex plane with $u\ne
v$. 
Then for all $k\in\NN$ we have
\begin{equation}\label{capacity}
\min_{w_1,\dots,w_k\in \CC} \max_{z\in J} \left| \prod_{j=1}^k
(z-w_j) \right| \ge 2 \left(\frac{|J|}{4}\right)^k~.
\end{equation}
\end{olemma}
\begin{proof}
Lemma \ref{l:capacity} is essentially the classical result of
Chebyshev for a real interval \cite{Chebyshev}, cf. \cite[Part
6, problem 66]{PS}, \cite{BE, MMR}. In fact, it holds for much
more general situations, e.g. it remains valid in exactly the
same form for arbitrary ${J \Subset \RR}$. Indeed, according to
\cite{Schiefermayr} we have $ \min\limits_{w_1,\dots,w_k\in\RR}
\max\limits_{z\in J} \left| \prod_{j=1}^k (z-w_j) \right| \ge 2~
\capa(J)^k$, which can be combined with the below result of
P\'olya, see Lemma \ref{l:transfinitediammeasure}, to get the
statement.
\end{proof}

Recall the well-known basic facts about the Chebyshev constant,
the transfinite diameter $\Delta(K)$, and logarithmic capacity
${\capa} (K)$, which coincide for all $K\Subset \CC$, as is
known from Fekete \cite{Fekete} and Szeg\H o \cite{Szego24}.
That also means a weak (i.e. logarithmic) asymptotical equality
of the quantities on the two sides of \eqref{capacity}.
However, even the same result as in \eqref{capacity} holds true
(perhaps with the unimportant loss of the factor 2) even for
complex compacta. We will use such type of estimations in the
following form.


\begin{olemma}{\bf(Faber, Fekete, Szeg\H o).}
\label{l:transfinitediam}
Let $M \Subset \CC$ be any compact set.
Then for all
$k\in\NN$ we have
\begin{equation}\label{transfdiam}
\min_{w_1,\dots,w_k\in \CC} \max_{z\in M} \left| \prod_{j=1}^k
(z-w_j) \right| \ge \Delta(M)^k=\capa(M)^k,
\end{equation}
where $\Delta(M)=\capa(M)$ is the transfinite diameter and the
capacity of the set~$M$.
\end{olemma}
\begin{proof} Regarding the formulation in Lemma \ref{l:transfinitediam}  cf.
Theorem 5.5.4. (a) in \cite{Rans} or \cite[(3.7) page
46]{SaffTotik}. Historically, it was first Fekete who proved
the inequality and also (with say $R=\CC$) that the left and
middle quantities are--logarithmically, i.e. after normalizing
by taking $k^{\rm th}$ roots--asymptotically equivalent.
Moreover, he showed that in cases when $\CC\setminus M$ is a
simply connected domain (if considered with the point
$\infty$), then the limits of the $k^{\rm th}$ roots also agree
to the so-called \emph{conformal radius} $\rho(M)$. Before
that, Faber \cite{Faber} has already proved $\max_{z\in M}
\left| \prod_{j=1}^k (z-w_j) \right| \ge \rho(M)^k$ for $M$ a
Jordan domain bounded by a closed analytic Jordan arc.
Following Fekete, Szeg\H o showed that the condition of
$\CC\setminus M$ being simply connected is not necessary, and
that with the so-called Robin constant $\gamma(M)$ (equivalent
to capacity), the stated inequalities hold true, moreover,
$\gamma(M)=\Delta(M)$ in general for all compacta.
\end{proof}

\begin{olemma}{\bf(P\'{o}lya, see~\cite{Goluzin}[Ch. VII]).}\label{l:transfinitediammeasure}
Let $J \subset \RR$ be any compact set, $|J|^*$ be its outer
Jordan measure. Then $|J|^*\le 4\Delta(J)$.
\end{olemma}
\begin{proof}
See \cite[Ch. VII]{Goluzin})\footnote{Of course, the Lebesgue
measure $|J|$ of the compact set $J\Subset \RR$ does not exceed
its outer Jordan mesure $|J|^*$}. To reflect back to the above
discussion of general forms of Chebyshev's Lemma
\ref{l:capacity}, recall that $\Delta(J)={\capa} (J)$.
\end{proof}

There are several known estimates for capacities and the
so-called ``Widom factors" (see e.g. \cite{Goncharov, Widom}
and the references therein) between Chebyshev constants and
corresponding powers of the transfinite diameter: for us, these
more precise estimates are not needed, as \eqref{transfdiam}
suffices. There are some known explicit computations or
comparisons and estimates of capacities from other geometric
parameters of the respective sets: a few most basic ones can be
found e.g. in the survey of Ransford \cite{RansSur} or in his
book \cite[page 135]{Rans}.

One remarkable fact to be noted also here is that the values of
the diameter and the transfinite diameter are within a constant
factor: for any compact set $E\subset \CC$ we have $\Delta(E)
\le \diam(E)/2$, and if $E$ is also connected, then $\Delta(E)
\ge \diam(E)/4 $, too, the disk $\DD$ and the interval $\II$
showing sharpness of both estimates, respectively. See e.g.
\cite[\S 1.7.1.]{RansSur}\footnote{However, note a disturbing
misprint in this fundamental reference: in \S 1.7.2. the first
two displayed formulas must be corrected to have the opposite
direction of the inequality sign.}.

\section{Some geometrical notions and definitions}\label{sec:geodef}

We start with a \emph{convex, compact domain} $K\Subset \CC$. Then its interior $\intt K \ne \emptyset$ and $K=\overline{\intt K}$, while its boundary $\Gamma:=\partial K$ is a convex Jordan curve. More precisely, $\Gamma = {\mathcal R}(\gamma)$ is the \emph{range} of a continuous, convex, closed Jordan curve $\gamma$ on the complex plane $\CC$.

If the parameter interval of the Jordan curve $\gamma$ is $[0,L]$, then this means, that $\gamma: [0,L] \to \CC$ is continuous, convex, and one-to-one on $[0,L)$, while $\gamma(L)=\gamma(0)$. While this is the most used setup for curves, we need the two, essentially equivalent interpretations i.e. this compact interval parametrization and also the periodically extended interpretation with $\gamma(t):=\gamma(t-[t/L]L)$ defined periodically all over $\RR$. If we need to distinguish, we will say that $\gamma:\RR\to\CC$ and $\gs : \TT:=\RR/L\ZZ \to \CC$, or equivalently, $\gs : [0,L] \to \CC$ with $\gs(L)=\gs(0)$.

Curves can be parameterized equivalently various ways. However, in this work we will restrict ourselves to parametrization with respect to arc length: as the curves are convex, whence rectifiable curves, they always have finite arc length $L:=|\gs|$, and parametrization is possible with respect to arc length. Whence also $$L:=|\gs|=|\Gamma|$$ is the arc length of $\Gamma$, i.e. the perimeter of $K$. The parametrization $\gamma: \RR \to \partial K$ defines a unique ordering of points, which we assume to be positive in the counterclockwise direction, as usual.

This has an immediate consequence also regarding the
derivative, which must then have $|\dot{\gamma}|=1$, whenever
it exists, i.e. (linearly) a.e. on $[0,L)\sim \TT$. Since
$\dot{\gamma} :\RR \to \partial \DD$, we can as well describe
the value by its angle or argument: the derivative angle
function will be denoted by $\alpha:=\arg \dot{\gamma} : \RR
\to \RR$. Since, however, the argument cannot be defined on the
unit circle without a jump, we decide to fix one value and then
define the extension continuously: this way $\alpha$ will not
be periodic, but we will have rotational angles depending on
the number of (positive or negative) revolutions, if started
from the given point. With this interpretation, $\alpha$ is an
a.e. defined nondecreasing real function with
$\alpha(t)-\frac{2\pi}{L} t$ periodic (by $L$) and bounded.
With the usual left- and right limits $\alpha_{-}$ and
$\alpha_{+}$ are the left- resp. right-continuous extensions of
$\alpha$. The geometrical meaning is that if for a parameter
value $\tau$ the corresponding boundary point is
$\gamma(\tau)=\ze$, then $[\alpha_{-}(\tau),\alpha_{+}(\tau)]$
is precisely the interval of values $\beta \in \TT$ such that
the straight lines $\{\zeta+e^{i\beta}s~:~ s\in \RR\}$ are
supporting lines to $K$ at $\zeta \in \partial K$. We will also
talk about half-tangents: the left- resp. right- half-tangents
are the half-lines emanating from $\zeta$ and progressing
towards $-e^{i\alpha_{-}(\tau)}$ and $e^{i\alpha_{+}(\tau)}$,
resp. The union of the half-lines $\{\zeta+e^{i\beta}s~:~ s\ge
0\}$ for all $\beta\in [\alpha_{+}(\tau),\pi+\alpha_{-}(\tau)]$
is precisely the smallest cone with vertex at $\zeta$ and
containing $K$.

We will interpret $\alpha$ as a multi-valued function, assuming all the values in $[\alpha_{-}(\tau),\alpha_{+}(\tau)]$ at the point $\tau$.
Restricting to the periodic (finite interval) interpretation of $\gs: [0,L)\to \CC$, without loss of generality we we may assume that $\as:=\arg(\dot{\gs}) :[0,L]\to [0,2\pi]$. In this regard, we can say that $\as:\RR/L\ZZ \to \TT$ is of bounded variation, with total variation (i.e. total increase) $2\pi$--the same holds for $\alpha:\RR\to\RR$ over one period.

The curve $\gamma$ is differentiable at $\zeta=\gamma(\theta)$ if and only if $\alpha_{-}(\theta)=\alpha_{+}(\theta)$; in this case the unique tangent of $\gamma$ at $\zeta$ is $\zeta+e^{i\alpha}\RR$ with $\alpha=\alpha_{-}(\theta)=\alpha_{+}(\theta)$.

It is clear that interpreting $\alpha$ as a function on the boundary points $\zeta\in \partial K$, we obtain a parametrization-independent function: to be fully precise, we would have to talk about $\gw$, $\gws$, $\aw$ and $\aws$. In line with the above, we consider $\aw$, resp $\aws$ \emph{multivalued functions}, all admissible supporting line directions belonging to $[\alpha_{-}(\tau),\alpha_{+}(\tau)]$ at $\zeta=\gamma(\tau)\in \DK$ being considered as $\aw$-function values at $\ze$.
At points of discontinuity $\alpha_{\pm}$ or $\as_{\pm}$ and similarly $\aw_{\pm}$ resp. $\aws_{\pm}$ are the left-, or right continuous extensions of the same functions.

A convex domain $K$ is called {\em smooth}, if it has a unique
supporting line at each boundary point of $K$. This occurs iff
$\alpha_{\pm}:=\alpha$ is continuously defined for all values
of the parameter. For a supporting line $\zeta+e^{i\beta}\RR$
the outer normal vector is ${\bnu}(\zeta):= e^{i\beta-i\pi/2}$,
and the (outer) normal vectors are precisely the vectors $\bnu$
satisfying $\Inner{z-\zeta}{\bnu} \le 0$ ($\forall z \in K$)
with the usual $\RR^2$ scalar product, or equivalently, $\Re
\left((z-\ze)\overline{\bnu}\right))\le 0$.

For obvious geometric reasons we call the jump function $\Omega:=\alpha_{+}-\alpha_{-}$ the {\em supplementary angle} function. This is identically zero almost everywhere (and in fact except for a countable set), and has positive values such that the total sum of the (possibly infinite number of) jumps does not exceed the total variation of $\alpha$, i.e. $2\pi$.

In the sequel we use some local quantities like the {\em local depth} $h_K(\ze)$ at boundary points $\ze\in\DK$. Take any boundary point $\zeta\in\partial K$, and a supporting line at $\zeta$ to $K$ with corresponding normal vector $\bnu=\bnu(\zeta)$. It is easy to see\footnote{For if with some normal $\bnu$ and the normal line $\ell:=\zeta+\bnu \RR$ we have $\ell \cap K =\{\zeta\}$, then (due to convexity) there cannot be interior points on both sides of this line; the same being true for the supporting line $t:=\zeta+i\bnu \RR$, we find that $\intt K\ne \emptyset$ lies strictly inside one quadrant of the plane, whence by the fatness of $K$ also $K$ is in one closed quadrant, and so to any interior point $w\in \intt K$ the direction of $\zeta-w$ is normal to $K$ at $\zeta$.}
that at least for some normal directions $\bnu$ we have
\begin{equation}\label{pointdepth}
\zeta+\bnu \RR \cap K = [\zeta,\zeta-h \bnu]~
\end{equation}
with some \emph{positive length} $h$ (unless $\intt
K=\emptyset$). Further, it is also easy to see that the
supremum of all such positive lengths in \eqref{pointdepth} is
actually a maximum. This maximum is denoted as $h:=h_K(\zeta)$,
and is called the {\em (local) depth} of $K$ at $\zeta$.

With this, we can as well define the (\emph{global}) depth of the convex domain $K$.
\begin{definition}\label{d:depth}
A convex body $K$ has {\em depth} $h_K$ with
\begin{equation}\label{eq:bodydepth}
h_K:=\inf \{ h_K(\zeta)~:~\zeta\in\partial K\} ~.
\end{equation}
The convex domain $K$ has {\em fixed depth} or {\em positive depth}, if $h_K>0$.
\end{definition}

Note that this quantity is not always a minimum, and it can as well be zero, as e.g. in case of the regular triangle.

It is easy to see that if a convex domain has a boundary point
$\zeta\in\DK$, where $\Omega(\ze)>\pi/2$, then the convex domain
cannot have positive depth. On the other hand there cannot be
too many of such boundary points, since the sum of the jumps of
$\alpha_{\pm}$ at these points cannot exceed the total variation
$2\pi$ of $\alpha_{\pm}$. So there exist at most three points
with obtuse supplementary angles, and at most four points with
$\Omega(\ze)>2\pi/5$, etc.. It is then clear that the largest
supplementary angle exists as the maximum of the nonnegative
function $\Omega$ over the boundary $\DK$. We can thus define
\begin{definition}\label{def:OmegaK}
For any convex domain $K$ the \emph{largest supplementary angle} is
$$
\Omega_K:=\max_{\DK} \Omega =\max_{\DK} (\alpha_{+}-\alpha_{-}).
$$
\end{definition}

Note that $\Omega_K=0$ if and only if $K$ is smooth. Also note that the Szeg\H o type outer angle of a say convex domain is $\Omega_K+\pi$.

Finally, let us introduce a version of the \emph{modulus of
continuity} function of the normal direction(s) $\arg \bnu
(\zeta)$ (with respect to distance, i.e. chord length) on the
boundary $\partial K$. Let $d_\TT(\theta,\tau)$ denote the
usual distance on the circle, i.e. the distance of
$\theta-\tau$ from $2\pi\ZZ$. Then the modulus of continuity of
the normal vectors of the boundary is defined the following
parametrization--free way.
\begin{definition}\label{def:omega}
The modulus of continuity of the (say: outer) normal directions
of the boundary of $K$ is defined for $0\le t <w$ as
\begin{equation}\label{eq:modcont}
\begin{aligned}
\omega_K(t) := \sup \left\{ \right. &d_\TT(\arg \bnu,\arg \bmu) \  :
\\ &\   \text{ $\bnu$, $\bmu$  are  normal  to $K$  at  $z$, $z'$ resp.},
 \left. \ |z-z'|\le t\right\}.
\end{aligned}
\end{equation}

This is equivalent to the modulus of continuity of the tangent directions $\alpha(\tau)$ with respect to chord length or distance, i.e.
$$
\omega_K(t)=\sup \{d_\TT(\as(\sigma),\as(\tau))~:~ \tau,\sigma\in\RR,\   |\gamma(\sigma)-\gamma(\tau)|\le t\}.
$$
Further, $\omega_K$ can be expressed the following (also
parametrization-free) way, too:
$$
\omega_K(t)=\sup
\{d_\TT(\aws(\ze),\aws(\ze'))~:~ \ze,\ze'\in\DK,\,
|\ze-\ze'|\le t\}.
$$
\end{definition}

Observe that by the definition of the width $w=w_K$,
 precisely when the chord length reaches $w$, then there are points $z,z'\in \DK$ with parallel supporting lines,
 i.e. with opposite normals, achieving $d_\TT(\bnu,\bnu')=\pi$.
 From that distance on, any definition of the modulus of continuity can only say that $\omega_K(t)=\pi$ for $t\ge w$ .

Note that $\omega_K:[0,w) \to [0,\pi)$ is a nondecreasing
function with possible jumps: in particular, if
$\omega:=\omega_K$ is not continuous, then
$\omega_K(0)=\Omega_K$ is the very first jump (compared to
$0$), and all jumps have to be between zero and $\Omega_K$. Now
defining
\footnote{Observe that now, because we have defined
$\omega_K$ with the $\le$ sign, we in fact have
$\omega_K(t)=\omega_{+}(t)$; defining the modulus of continuity
with respect to the condition $|\ze-\ze'|<t$ would provide
$\omega_{-}(t)$.}
$$
\omega_{-}(0):=0,  \quad \omega_{+}(w)=\omega_L(w)=\pi,
$$
$$
\omega_{-}(t):=\sup_{s<t}\omega_K(s)=\limsup_{s\to t-0}
\omega_K(s)=\lim_{s\to t-0} \omega_K(s),
$$
and
$$
\omega_{+}(s):=\inf_{s>t} \omega_K(s)= \liminf_{s\to
t+0}\omega_K(s)=\lim_{s\to t+0}\omega(s) \quad (s<w),
$$
we can consider the modulus of
continuity function a multivalued function with
$\omega_K(s)=[\omega_{-}(s),\omega_{+}(s)]$. This way
$\omega_K$ becomes a surjective mapping from $[0,w]\to[0,\pi]$
(the full set of possible distances on $\TT$) and, again
allowing a multivalued interpretation, its inverse can be
defined as $\omega^{-1}(\sigma):=\{s \in [0,w]~:~
\omega_K(s)=\sigma\}$--again, in general, a multivalued
function mapping $[0,\pi]$ surjectively to $[0,w]$.

A version of the modulus of continuity \emph{on the boundary
curve} can also be considered according to arc length, or,
equivalently, according to parametrization: this will be the
ordinary modulus of continuity of the composite function $\bnu
\circ \gamma$, interpreting $\bnu$ as a multi-valued function
assuming \emph{all admissible values} of outer normal
directions; equivalently, the modulus of continuity of
$\alpha=\aw\circ \gamma$ on $\RR$. Note that--as we
parameterize the boundary curve $\gamma$ with respect to arc
length--the arc length of the subarc of $\gamma$ in the
counterclockwise direction between points corresponding to
parameter values $\tau<\sigma$ is $\int_\tau^\sigma |d\gamma|
=\tau-\sigma$.
\begin{definition}\label{def:omegagamma} The modulus of continuity of normal directions (or tangent directions) \emph{with respect to arc length on the boundary curve} is
\begin{equation*}
\omega_{\gamma}(s):=\omega(\bnu\circ\gamma;s)=\omega(\aw\circ\gamma;s)=\omega(\alpha;s)=\sup \{\alpha(\sigma)-\alpha(\tau)~:~ \tau\le\sigma\le \tau+s\}.
\end{equation*}
\end{definition}

Note that $\omega(\alpha,s): [0,L)\to [0,2\pi)$ monotonically.
The main difference between the two definitions is lying not in
the different measurement of the distances between the points--
which is already an essential difference, though--but in the
fact that the distance of values is measured not in $\TT$, but
in $\RR$, thus allowing $\omega_\gamma$ to increase all over
$[0,\infty)$. Still, after reaching the period $L$ there is not
much point to consider this modulus, for there we only add the
number of full revolutions times $2\pi$, i.e. we have
$\omega_\gamma(t)=\omega_\gamma(t-[t/L]L)+[t/L]2\pi$.

\section{A result for domains with positive depth}\label{sec:posdepth}

\subsection{Statement of the result} The aim of this section is to prove
the following theorem on the order of the oscillation in $L^q$
norm for a class of domains defined by the property of having
positive depth.

\begin{theorem}\label{th:posdepth}
Assume that $K\Subset \CC$ is a convex domain with positive depth $h_K>0$. Then for any $n\in \NN$ and $p\in\PK$ it holds
\begin{equation}\label{eq:ordernLq}
\|p'\|_{q,K} \ge \frac{h_K^4}{3000 d^5} ~n ~\|p\|_{q,K} ~.
\end{equation}
\end{theorem}

\begin{remark}
Note the generality of the statement. E.g. a regular $k$-gon
$G_k$ is always of positive depth from $k=4$ on; the only (and
essential) exception being the regular triangle. More
generally, a polygonal domain has fixed depth iff it has no
acute angles.
\end{remark}

The whole section is devoted to the proof of this result. In
the course of proof we will work out various intermediate steps
and results, which will be used later.

In Section \ref{sec:osciupper} we will show that this order of oscillation is again the best possible, like in the case $q=\infty$, thus settling the question of the order of the oscillation in $L^q$ norms for all these domains.

Before starting the main argument, we first give a geometrical discussion of the property that $h_K>0$.

\subsection{A characterization of fixed depth} \label{ss:fixeddepthchar}

We have already remarked that $\Omega_K>\pi/2$ implies that $h_K=0$. When $\Omega_K=\pi/2$, the situation is a little ambiguous.
\begin{remark} Observe that in particular in case $\Omega_K=\pi/2$ a rectangle $R$ has positive depth, but the upper semi-disk $U:=\{ z \ : \  |z|\le 1 \quad {\rm and} \quad \Re z \ge 0 \}$ admits zero depth. This is easy to see directly, observing that local depths of boundary points on the diameter tend to zero as the points approach the vertices at $\pm 1$. In particular, in case $\Omega_K=\pi/2$, both $h_K>0$ and $h_K=0$ can occur.
\end{remark}

Still, a precise characterization of positive depth is possible. The following must be well-known in geometry, but finding no reference to that, we decided to describe this characterization also here.

\begin{proposition}\label{p:anglesmooth}
Let $K$ be any convex domain. Then there are the following cases.
\begin{enumerate}[(i)]
\item If for all boundary points the supplementary angles $\Omega(\zeta)$ admit $\Omega(\zeta)<\pi/2$,
-- that is, if $\Omega_K<\pi/2$ -- then the domain $K$ has a fixed positive depth $h_K>0$.
\item If for some boundary point(s) the supplementary angle(s) satisfy $\Omega(\zeta)>\pi/2$, then $h_K=0$.
\item If $\Omega_K =\pi/2$, but at each boundary point with $\Omega(\zeta)=\pi/2$ the tangent angle function $\alpha$ is constant $\alpha_{-}(\zeta)$ resp. $\alpha_{+}(\zeta)$ in a small left, resp. right neighborhood of $\zeta$ -- i.e. if the point $\zeta\in \partial K$ is a vertex, with two orthogonal straight line segment pieces of the boundary joining at $\zeta$ -- then the domain $K$ has a fixed positive depth $h_K>0$.
\item If $\Omega_K=\pi/2$, but there exists a maximum point of $\Omega$, in any neighborhood of which either the left or the right neighboring piece of the boundary fails to be a straight line segment, then $h_K=0$.
\end{enumerate}
\end{proposition}

\begin{corollary}\label{cor:smooth} If $K$ is smooth--i.e. $\DK$ is a smooth convex Jordan curve--then $h_K>0$.
\end{corollary}

\begin{proof} First note that Cases (i)--(iv) indeed give a full list of possibilities.

Let us first consider Case (i). We argue by contradiction. Take any boundary points
$\zeta_n$ with corresponding normal vectors $\bnun$ and satisfying
\begin{equation}\label{fixdepth}
\zeta_n+\bnun \RR \cap K = [\zeta_n,\zeta_n-h_n \bnun]~.
\end{equation}
with $h_n<1/n$. Let $\omega_n:=\zeta_n-h_n \bnun$ and any
corresponding outer normal vector be $\bmun$. After selecting a
subsequence, if necessary, by compactness we may assume that these
points and vectors converge. Hence let $\zeta_n\to\zeta$,
$\bnun\to \bnu$ and $\bmun\to \bmu$. Since $h_n\to 0$, it follows
that $\omega_n\to\zeta$, too.

Observe that for any point $z\in K$ normality of $\bnun$ at
$\zeta_n$ means $\Inner{\bnun}{z-\zeta_n}\le 0$, and normality of
$\bmun$ at $\omega_n$ means $\Inner{\bmun}{z-\omega_n}\le 0$,
hence by the above convergence we must have
$\Inner{\bnu}{z-\zeta}\le 0$, and also $\Inner{\bmu}{z-\zeta}\le
0$. In other words, both $\bnu$ and $\bmu$ are normal vectors at
$\zeta\in\partial K$.

However, $\zeta_n-\omega_n=h_n\bnun$ is parallel to $\bnun$, hence
normality of $\bmun$ at $\omega_n$ yields
$\Inner{\bmun}{\bnun}\le 0$. Again by continuity this entails
$\Inner{\bmu}{\bnu}\le 0$, that is, the angle between these two
normal vectors is at least $\pi/2$. Clearly then
$\Omega(\zeta)\ge\pi/2$, a contradiction to our assumption. This concludes the proof of Case (i).

Case (ii) is obvious, and even (iv) can be proved easily, but we give the proof here. So let us take a point $\zeta\in \partial K$ with the given property, and assume, as we may, that $\zeta=0$, and $\alpha_{-}(\zeta)=3\pi/2$, and $\alpha_{+}(\zeta)=0$, i.e., the two extremal half-tangents at $\zeta$ are the (positively directed) imaginary axis and the (positively directed) real axis. Also let us assume e.g. that no non-degenerate segment piece of the positive imaginary half-axis (i.e. the left half-tangent) belong to the boundary.

Put $a:=\max \{\Re z~:~z\in K\}$. It can happen that there are several points of $K$ where this maximum is attained; however, for $0$ the condition that the imaginary axis is a tangent and that no straight line piece of it belongs to $K$, implies that $K \cap i\RR = \{0\}$. Consider now any $x\in [0,a]$, and the intersection $\{\Re z =x \} \cap K$. This vertical segment has to be above (not below) the positive real axis, since $\RR$ is a tangent with $\alpha_{+}(0)=0$, and so we can write $\{\Re z =x \} \cap K = [x+ig(x),x+if(x)]$ with two continuous, nonnegative functions $0\le g \le f$ satisfying $g(0)=f(0)=0$ and $g$ convex, $f$ concave (i.e. "convex from below"), both continuous and in particular $\lim_{x\to 0+0} f(x) =0$, moreover, both $g$ and $f$ are nondecreasing in a right neighborhood of $0$.

Now let us show that at $z:=x+ig(x)$ the normal vector $\bnu$ has angle $\arg \bnu \in [3\pi/2,2\pi) $. First, $0\in K$ implies that $\langle -z,\bnu \rangle \le 0$, entailing $\arg \bnu \ge \arg (-z) + \pi/2 =\arg z + 3\pi/2 \ge 3\pi/2$. Then again, $\arg \bnu$ achieves $2\pi$ only when the boundary point is in rightmost position (i.e. of maximal real part value) in $K$, and thus $\arg \bnu \in [3\pi/2,2\pi)$ for $0<x< a$. So these mean that the normal line $z+\bnu \RR$ to $z=x+ig(x)$ intersects the boundary of $K$ at some point $z':=x'+if(x')$ with some $0 \le x'\le x$, and thus $f(x') \le f(x) \to 0$ implies that the length of intersection of $K$ and this normal line tends to 0 together with $x$. This completes the proof of $h_K=0$.

Finally, Case (iii) is again easy to prove. First, a little thought shows that the condition is equivalent to the statement that the modulus of continuity satisfies $\omega_+(0)=\pi/2$ and $\omega_K(t) = \pi/2$ for all $0\le t \le t_0$ with some positive value $0<t_0<w$.

So we prove now that if $\omega(t_0)=\pi/2$, or, more generally, if $\pi/2\in \omega(t_0)$ for some $t_0>0$, then for any boundary point $z\in\DK$ and any normal line $m:=z+\bnu\RR$ to $K$ at $z$, the length of the intersection of $m\cap K =[z,z']$ is at least $t_0$ (and so even $h_K\ge t_0 >0$).

Obviously, the chord vector between $z$ and $z'$ is in inner normal direction, whence has an angle (argument) exactly $\pi/2$ above the angle of the positively directed tangent, orthogonal to $\bnu$ at $z$.

Assume, as we may, that $z=0$, $z'=i y'$ (with some $y'>0$) and $m$ is just the imaginary axis. By definition of the width $w=w(K)$, as $w>t_0$, there exists a point $z_0=x_0+y_0\in K$ with $y_0\ge w$: if $x_0=0$, then we would have $z_0=iy_0 \in [z,z']=[0,y']$ and we would get $y'\ge w >t_0$, concluding the argument. So it remains to deal with $0<y'<w$ and $x_0 \ne 0$. Let e.g. $x_0<0$: then the three points $z ,z', z_0$ cut the boundary $\Gamma$ of $K$ into three parts, each of them extending between two of them and not containing the third one, and following in the counterclockwise direction as $z\prec z' \prec z_0$ and $\Gamma(z,z') \prec \Gamma(z',z_0) \prec \Gamma(z_0,z)$. According to these normalizations, $\alpha(z') \le \arg(z_0-z') \in (\pi/2,\pi)$, for the latter chord vector is $x_0+i(y_0-y')$ with $x_0<0$ and $y_0\ge w >y'$. It follows that for any further points $z^{*} \in \Gamma(z',z_0)$ we necessarily have $\pi\ge \al(z^{*}) \ge \al_{+}(z')$. Note that (unless $z^{*} = z'$) we cannot have $z^{*}\in m$, but only $\Re z^{*} <0$ and $\arg(z^{*}-z) > \arg(z'-z)=\pi/2$. Whence $\pi \ge \alpha(z^{*}) \ge \arg(z^{*}-z) > \pi/2$, so that $\omega_K(|z^{*}-z|)>\pi/2$. As now $z^{*}$ can be arbitrarily close to $z'$, for any value $t^{*}>|z'-z|=y'$ we have $\omega_K(t^{*})>\pi/2$, i.e., any such $t^{*}$ must satisfy $t^{*}>t_0$. So if $t^{*}>y'$ then we have $t^{*}>t_0$, whence $|z'-z|=y'\ge t_0$, and $h_K\ge t_0$, as needed.
\end{proof}

\begin{proposition}\label{p:omegaandh} For a convex, compact domain $K$ we have $h_K>0$ if and only if the modulus of continuity function has $\omega_K(\tau)\le \pi/2$ for all $0\le \tau \le t$ with some $t>0$.

Furthermore, with the above extended interpretation of the inverse function of the modulus of continuity function, we have $h_K\ge\mu_K:=\max \omega^{-1}(\pi/2)$.
\end{proposition}
\begin{proof} Consider the four cases in the above Proposition \ref{p:anglesmooth}. A little thought shows that Case (i) is equivalent to state $\omega_K(\tau) <\pi/2$ ($0\le \tau \le t$) for sufficiently small $t>0$, and (iii) is exactly the case when $\omega_K(\tau) =\pi/2$ ($0\le \tau \le t$) for sufficiently small $t>0$. So these cases with $h_K>0$ are such that $\omega_K(\tau)\le \pi/2$ for all $0\le \tau \le t$ with some $t>0$. Moreover, Case (ii) means $\omega_K(0)>\pi/2$, and Case (iv) means that $\omega_K(0)=\pi/2$, but for all $\tau>0$ already $\omega_K(\tau)>\pi/2$, whence the cases with $h_K=0$ are the ones with $\omega_K(\tau)>\pi/2$ for all $\tau>0$. This proves the first assertion of the Proposition.

As for the last assertion, there is nothing to prove for $\mu_K=0$, so we may take $\mu_K>0$, meaning that there are points $0<t_0 \in\omega^{-1}_K(\pi)$. Recalling the last argument of the proof of the previous Proposition \ref{p:anglesmooth}, we can then see that for any such $t_0$ also $h_K\ge t_0$ holds. Taking $t_0:=\max \omega_K^{-1}(\pi/2)$ thus provides $h_K\ge \mu_K$, too.
\end{proof}

\begin{remark}
To see that the inequality
$h_K > \omega_K^{-1}(\pi/2)$ is possible, it suffices to
consider a regular hexagon $G_6$ of side length $h$, say. Then
$h_K(G_6)=2h$, while $\omega_{-}(h)=\pi/3$,
$\omega_{+}(h)=2\pi/3$ and $\omega_K^{-1}(\pi/2)=\{h\}$,
$\mu_K=h$.
\end{remark}

\subsection{Technical preparations for the investigation of $L^q(\partial K)$ norms}\label{ss:Nikolski}

First, we will prove a Nikolskii-type estimate, which is
similar to the well-known analogous inequality on the real
line, found in the book of Timan~\cite[4.9.6 (36)]{Timan}.

\begin{lemma}\label{l:Nikolskii} For any polynomial of degree at most $n$ we have that
\begin{equation}\label{eq:Nikolskii}
\|p\|_{L^q(\DK)} \ge  \left(\frac{d}{2(q+1)}\right)^{1/q} ~\|p\|_{L^\infty(\DK)} ~ n^{-2/q} .
\end{equation}
\end{lemma}
\begin{proof} We first prove a Bernstein-Markov type estimate in the maximum norm.
Let $z\in K$ be arbitrary. Then we always have a chord $J:=[z,z^{*}] \subset K$ of length at
least $d/2$, for we can take any diameter $I$ of $K$, and for $z^{*}$ take the endpoint of $I$,
which is situated farthest from $z$--which is of course at least of distance $d/2$ from $z$.
Applying Markov's Inequality on $J$, we obtain
$$
|p'(z)| \le \dfrac{n^2}{|J|/2} \|p\|_{L^\infty(J)} \le \dfrac{4n^2}{d} \|p\|_{L^\infty(\DK)}.
$$
Therefore, it holds
\begin{equation}\label{eq:Bernstein}
\|p'\|_{L^\infty(\DK)} \le \frac{4}{d} n^2 \|p\|_{L^\infty(\DK)}.
\end{equation}

Consider now a point $z_0=\gamma(t_0)\in \DK$, where $\|p\|_\infty$ is attained. Then
-- using also convexity of $K$ -- at each point $z\in K$ we have
\begin{equation}\label{eq:Bernsteinpchange}
\begin{aligned}
|p(z)| &=\left|p(z_0)+\int_{z_0}^{z} p' (\ze)d\ze\right| \geq \|p\|_\infty- |z-z_0| \|p'\|_\infty
\\&\geq \|p\|_\infty\left(1 - \frac{4|z-z_0|}{d}n^2\right).
\end{aligned}
\end{equation}
Now, from \eqref{eq:Bernsteinpchange} we can estimate the $q$-integral of $p$ as follows. Take $r_0=d/(4n^2)$, $U:=D(z_0,r_0)$ and $\Gamma_0:=\Gamma\cap U$. More precisely, in case there are several pieces of arcs in this intersection (which can only happen for small $n$, though) then we take only one arc which passes through the point $z_0$ and extends to the circumference of $U$ in both directions -- and drop the remaining pieces. Let us denote the points, falling on the circumference $\partial U$ right preceding and following $z_0=\gamma(t_0)$ on $\gamma$ as $z_{\pm}:=\gamma(t_{\pm})$. Recalling that $\gamma$ is parameterized according to arc length, and writing for the parametrization $\gamma: [t_{-},t_{+}] \to \Gamma_0$, and so in particular $\gamma_{-}:[t_{-},t_0]\to \Gamma_{-}$ and $\gamma_{+}:[t_0,t_{+}]\to \Gamma_{+}$,
we get
\begin{align*}
\int_{\gamma_{+}} |p|^q |d\gamma| &\geq \int_{t_{0}}^{t_{+}} \left( 1 -
\frac{4|\gamma(t)-z_0|}{d}n^2 \right)^q \|p\|_\infty^q  dt
\\&\geq\|p\|_\infty^q \int_{t_{0}}^{t_{+}} \left( 1 - \frac{4|\gamma_{[t_0,t]}|}{d}n^2 \right)_{+}^q  dt
\\&=\|p\|_\infty^q \int_{t_{0}}^{t_{0}+\frac{d}{4n^2}} \left( 1 - \frac{4(t-t_0)}{d}n^2 \right)^q  dt
=\|p\|_\infty^q \frac{d}{4n^2} \int_{0}^{1} \left( 1 - s \right)^q  ds
\end{align*}
and similarly for $\gamma_{-}$, whence
$
\int_{\gamma_{0}} |p|^q |d\gamma| \geq  \|p\|_\infty^q \dfrac{d}{2(q+1)n^2} .
$
As a result, we get
$$
\|p\|_{L^q(\DK)} \ge \left(\int_{\gamma_0} |p|^q |d\gamma| \right)^{1/q}
\geq \left( \frac{d}{2(q+1)}\right)^{1/q} \|p\|_\infty  n^{-2/q},
$$
and the result follows.
\end{proof}

Next, let us define the subset $\HH:=\HH_K^q(p) \subset \DK$ the following way.
\begin{equation}\label{eq:Hsetdef}
\HH:=\HH_K^q(p):=\{\zeta\in \partial K~: ~ |p(\zeta)| > c n^{-2/q} \|p\|_\infty\},
\qquad c:=\left(8\pi(q+1) \right)^{-1/q} .
\end{equation}
Then we can restrict ourselves to the points of $\HH$
and neglect whatever happens for points belonging to
 $$
\HH^c=\Gamma \setminus \HH.
$$
Indeed, $\Gamma$ is contained in a disk of radius $d$ around any point of $K$,
 whence by the well-known property\footnote{A reference is \cite[p. 52, Property 5]{BF} about surface area,
presented as a consequence of the Cauchy Formula for surface area.} of convex curves,
$L:=|\gamma|\le 2\pi d$, and the above Lemma \ref{l:Nikolskii} furnishes
$$
\int_{\Gamma\setminus \HH} |p(z)|^q |dz| \le
  \frac{ 2\pi d c^q}{ n^{2} } \|p\|^q_\infty \le 4\pi(q+1) c^q  \|p\|_q^q =
  \frac{1}{2} \|p\|_q^q.
$$
That leads to
$$
\int_{\HH} |p|^q |d\gamma| = \int_{\gamma} |p|^q |d\gamma|  - \int_{\gamma\setminus \HH} |p|^q |d\gamma| \ge \|p\|_q^q - \frac{1}{2} \|p\|_q^q \ge \frac{1}{2} \|p\|_q^q.
$$
Therefore we can restrict to (lower) estimations of $|p'(\ze)|$ on the set $\HH$ where $p$ is assumed to be relatively large (compared to its maximum norm), so that we can assume that
$$
\log\dfrac{\|p\|_\infty}{|p(\ze)|} \le \log ( c^{-1} n^{2/q} ) =
\frac{\log (1+q)}{q} + \frac{\log\left(8\pi\right)}{q}
+ \frac2{q} \log n \le \log(16\pi) + 2 \log n .
$$
for all $q\ge 1$ and $n \ge \sqrt{16\pi}$, i.e. already for $n\ge 8$. Summing up we have

\begin{lemma}\label{l:Hlogp} Let $\HH \subset \DK$ be defined according to \eqref{eq:Hsetdef}. Then for all $p \in \PP_n$ we have
\begin{equation}\label{eq:pqintegralonH}
\int_{\HH} |p|^q \geq \frac12 \|p\|^q_{L^q(\DK)}.
\end{equation}
Furthermore, for any point $\ze \in \HH$, and for any $p\in \PK$ 
we also have
\begin{equation}\label{eq:pnormperponH}
\log \frac{\|p\|_\infty}{|p(\ze)|} \le \log(16\pi) + 2 \log n ~~
(\forall n \in \NN).
\end{equation}
\end{lemma}

For relatively small values of the degree $n$ we may get better
constants if using an entirely different argument, not suitable
to analyze the order regarding the degree, but yielding better
numerical values for small $n$. We will base our calculation on the following
classical result of Gabriel \cite[Theorem 5.1]{Gabriel}.

\begin{olemma}{\bf(Gabriel).}\label{Glemma}
If $\Gamma$ is any convex (closed) curve and $C$ any convex
curve inside $\Gamma$, and if $F(Z)$ is regular inside and on
$\Gamma,$
$$
\int_C |F^\lambda(Z)| |dZ| \le (\pi(e+1)+e)\int_\Gamma |F^\lambda(Z)| |dZ| ,
\quad ( \lambda\ge 0).
$$
\end{olemma}

With this we can now state the next.

\begin{lemma}\label{l:Gabrielyield}
For any point $\ze \in \partial K$ on the boundary of the
convex domain $K$ and $q\ge 1$ we have the estimate
\begin{equation}\label{eq:pzetaGabriel}
|p(\zeta)| \le d(\pi (\pi(e+1)+e))^{1/q} \|p'\|_q.
\end{equation}
As a direct consequence, we also have
\begin{equation}\label{eq:qtoqGabriel}
\|p'\|_q > \frac{1}{45.3} ~
\frac{1}{d} ~ \|p\|_q > 0.022 ~ \frac{1}{d} ~ \|p\|_q.
\end{equation}
\end{lemma}
\begin{proof} Evidently, it is enough to consider the case $1<q<\infty$.
Let $z_0\in K$ be any zero of $p$. Then for any $\zeta \in
\Gamma:=\DK$ the interval $[z_0,\zeta]\subset K$ (by
convexity), and so H\"older's inequality furnishes
\begin{align*}
|p(\zeta)|=\left|
\int_{[z_0,\zeta]} p'(z)dz\right|
\le \left(\int_{[z_0,\zeta]} |dz|\right)^{(q-1)/q}
\left(\int_{[z_0,\zeta]} |p'(z)|^q|dz|\right)^{1/q}.
\end{align*}
Noting that $|\zeta-z_0|\le d$ and applying Lemma~\ref{Glemma}
to the integral of $|p'(z)|^q$ over the convex curve
$C:=[z_0,\zeta] \cup [\zeta,z_0]$ (the degenerate closed convex
curve encircling around the points $\ze$ and $z_0$) we get
\begin{align*}
|p(\zeta)|& \le  d^{(q-1)/q}  \left(\frac12 (\pi(e+1)+e) \int_{\Gamma} |
p'(z)|^q|dz|\right)^{1/q}
\\&= d^{(q-1)/q}\left(\frac{\pi(e+1)+e}{2}\right)^{1/q} \|p'\|_q,
\end{align*}
proving \eqref{eq:pzetaGabriel}.

Now integrating on $q$-th power, and using again $L\le 2\pi d$ leads to
$$
\|p\|_q^q \le d^{q-1} ~\frac{\pi(e+1)+e}{2} \|p'\|_q^q ~ L \le
d^q~\pi (\pi(e+1)+e) \|p'\|_q^q < 45.3 ~d^q~\|p'\|_q^q.
$$
Taking $q$-th root, a small rearrangement finally furnishes even the last
assertion, as from here
$$
\|p'\|_q > \left(\frac{1}{45.3}\right)^{1/q}
\frac{1}{d} ~ \|p\|_q \ge \frac{1}{45.3} ~
\frac{1}{d} ~ \|p\|_q > 0.022 ~ \frac{1}{d} ~ \|p\|_q.
$$
\end{proof}

\subsection{Proof of a local result in terms of the local depth}\label{s:localdepthestimate}

In the following we will work out an unconditional pointwise estimate in
the sense that it will provide an estimate locally at points
$\ze \in \HH$ in terms of $h=h(\ze,K)$, not using the
assumption that $h_K=\inf_{\ze\in \DK} h(\ze,K)$ stays positive
or not. When this happens to hold, the below result will almost
immediately imply Theorem \ref{th:posdepth}.

\begin{theorem}\label{th:localdepth} Let $p\in\PK$ and let $\HH=\HH_K^q(p)$ be defined by \eqref{eq:Hsetdef}. Then it holds
\begin{equation}\label{eq:ptwisenh}
|p'(\ze)| \ge  \frac{h^4}{1500 d^5}~n~|p(\ze)|  \quad \left( ~{\rm if}~~\ze \in \HH
\right).
\end{equation}
\end{theorem}

\begin{proof}
First of all, we may assume that
$$
n\ge n_0:=32 ~ d^4/h^4,
$$
for values of $n$ not exceeding this bound we can settle the issue
referring to \eqref{eq:qtoqGabriel} of Lemma
\ref{l:Gabrielyield} providing
\begin{equation}\label{nsmall1500}
\|p'\|_q \ge \frac{n}{32d^4/h^4} \|p'\|_q  > \frac{1}{1500} ~\frac{h^4}{d^5} ~n~\|p\|_q.
\end{equation}

So from here on let us assume $n\ge 32 ~d^4/h^4$. Without loss
of generality we also assume $\zeta=0$ and that a tangent line
at $\zeta=0$, chosen according to the requirement
$h=h(\ze,K)(\ge h_K)$, is just the real line $\RR$ (and thus
the normal line is the imaginary axis $i\RR$). Then we have
$K\subset \HP:=\{{z\in\CC} \ : \ \Im z \ge 0\}$ and $K\cap
i\RR=[0,ih]$.

Let us denote the set of zeroes of $p\in\PK$ as
 $$
\Z:=\{z_j=r_je^{i\varphi_j}\ : \ j=1,\dots,n\}\subset K
$$
(listed with possible repetitions according to their multiplicity). We assume, as we may, that $p(z)=\prod_{j=1}^n(z-z_j)$.

In what follows we will use for any interval $[\sigma,\theta]$ (and similarly for $[\sigma,\theta)$ etc.) the following notations for the angular sectors, the zeroes in the angular sectors, and the number of zeroes in the angular sectors:
\begin{align*}
S[\sigma,\theta]&:=\{z\in\CC~:~ \arg z \in [\sigma,\theta]\}, \notag \\
\Z[\sigma,\theta]&:=\{z_j\in\Z~:~ \arg z_j \in [\sigma,\theta]\} = \Z\cap S[\sigma,\theta], \\
n[\sigma,\theta]&:=\# \Z[\sigma,\theta]. \notag
\end{align*}

Let us fix the angle
\begin{equation*}
\varphi:=\arcsin \left(\frac{h}{8d} \right)~.
\end{equation*}

We partition the zero set $\Z$ into two subsets as follows.
\begin{equation*}
\Z_{+}:=\Z[\varphi,\pi-\varphi]\,\qquad \Z_{-}:=\Z\setminus \Z_{+}~.
\end{equation*}
For the corresponding cardinals we write
\begin{equation*}
k:= \# \Z_{+}=n[\varphi,\pi-\varphi],\qquad m:=\# \Z_{-}=n[0,\varphi)+n(\pi-\varphi,\pi]~.
\end{equation*}
Observe that for any subset $\W\subset\Z$ we have
\begin{equation}\label{Impartsum}
\left| \frac{p'}{p}(0) \right| \ge -\Im {\frac{p'}{p}(0)} =
\sum_{j=1}^n \Im \frac {-1}{z_j} \ge \sum_{z_j\in\W} \Im \frac
{-1}{z_j}=\sum_{z_j\in\W} \frac {\sin \varphi_j}{r_j} \,,
\end{equation}
because all terms in the full sum are nonnegative. We apply inequality \eqref{Impartsum} with $\W=\Z_{+}$ to obtain
\begin{equation}\label{upperpart}
M := \left| \frac{p'}{p}(0) \right| \ge \sum_{z_j\in\\Z_{+}} \frac {\sin \varphi_j}{r_j} \ge {\sin\varphi}\sum_{z_j\in\\Z_{+}}\frac1{r_j} \ge
\frac{h}{8d^2} k
\,,
\end{equation}
since for $z_j\in\Z_{+}$ we have $\varphi_j\in [\varphi,\pi-\varphi]$,
$\sin \varphi_j \ge \sin\varphi=h/(8d)$, and $r_j\le d$.

Now put
\begin{equation*}
J:= \left[\frac 34 ih, ih \right]\subset K ~.
\end{equation*}
We estimate the distance of any $z_j=x_j+iy_j\in\Z_{-}$ from $J$.
In fact, taking any point $z=x+iy=re^{i\psi}\in D_d(0)\cap
(S[0,\varphi)\cup S(\pi-\varphi,\pi])$, we necessarily have
$|z|^2=x^2+y^2\le d^2$, \ $0\le y\le d \sin \varphi = h/8$, and therefore,
$\dist (z,J)=|z-i3h/4|=\sqrt{x^2+(y-3h/4)^2}$.
Clearly, then
\begin{equation}\label{distquotient}
\frac{\dist(z,J)^2}{|z|^2}=\frac{x^2+y^2-3yh/2+(3h/4)^2}{x^2+y^2}\ge
1+\frac{3h^2}{8d^2}~.
\end{equation}


In fact, we can do a little better here, taking into account
that the other endpoint of $J$, $ih$, also belongs to $K$,
whence the diameter provides an upper bound to $|z-ih|$, too,
yielding $x^2+(y-h)^2 \le d^2$. Using this in the middle of
\eqref{distquotient}, we may write
\begin{align*}
\frac{\dist(z,J)^2}{|z|^2}
&=\frac{x^2+y^2-3yh/2+(3h/4)^2}{x^2+y^2}
\\&=
1+\frac{9h^2-24yh}{16(x^2+y^2)} \ge
1+\frac{9h^2-24yh}{16(d^2-h^2+2yh)}~,
\end{align*}
where the last expression is decreasing in $y\le h/8$, whence admitting
$$
\frac{9h^2-24yh}{16(d^2-h^2+2yh)} \ge
\frac{9h^2-24(h/8)h}{16(d^2-h^2+2(h/8)h)} = \frac{3h^2}{8d^2-6h^2}.
$$
Introducing the parameter
$u:=d/h \in [1,\infty)$ we thus obtain
\begin{equation*}
\frac{\dist(z,J)^2}{|z|^2} \ge 1 + 
\frac{3h^2}{8d^2-6h^2} = \frac{8u^2-3}{8u^2-6} \qquad \left( u:=\dfrac{d}{h} \in [1,\infty)~\right).
\end{equation*}
From here taking logarithms we get
\begin{equation}\label{eq:zjtaufracu}
\left|\frac{z_j-\tau}{z_j} \right| \ge \exp\left( \frac12 \log\left( \frac{8u^2-3}{8u^2-6} \right) \right) \qquad \left( u:=\frac{d}{h} \in [1,\infty)~\right) ~.
\end{equation}

Next consider the set $R:=K \cap S[\varphi,\pi-\varphi]$. Applying Lemma \ref{l:capacity} to $R$
and $J \subset R$ we are led to
\begin{equation}\label{capacitylower}
\begin{aligned}
&\max_{z\in J} \prod_{z_j\in\Z_{+}}\left|\frac{z_j-z}{z_j}\right|
\ge \frac{1}{d^k} \max_{z\in J} \prod_{z_j\in\Z_{+}}|z_j-z| \ge \frac{1}{d^k}
\left(\frac{|J|}{4} \right)^k
\\
& =  
\exp\left(-k \log\left( \frac{16d}{h} \right) \right).
\end{aligned}
\end{equation}
Taking now the point $z_0\in J$ where this maximum (i.e. the
maximum of $|p(z)|$ on $J$) is attained, combining
\eqref{eq:zjtaufracu} and \eqref{capacitylower} and using $m+k=n$ leads to
\begin{align*}
\left| \frac{p(z_0)}{p(0)} \right| &= \prod_{z_j\in\Z}\left|\frac{z_j-z_0}{z_j}\right|
\ge \exp\left( m \frac12 \log\left( \frac{8u^2-3}{8u^2-6} \right) - k \log\left( 16u \right) \right) 
\\ & = 
\exp\left(n\frac{1}2 \log\left( \frac{8u^2-3}{8u^2-6} \right) - \left\{\frac12 \log\left( \frac{8u^2-3}{8u^2-6} \right) + \log\left( 16u \right) \right\} k \right)
\\ & =\exp\left(\frac{n}2 \log\left( \frac{8u^2-3}{8u^2-6} \right) - \psi(u) k \right),  \notag
\end{align*}
where
$$
\psi(u):=\frac12 \log\left( \frac{8u^2-3}{8u^2-6}\right)+\log(16u) .
$$


Taking logarithm, dividing by $\psi(u)$ and rearranging thus provides
\begin{align*}
k &\ge \frac{1}{\psi(u)}\left( \frac{n}2 \log\left( \frac{8u^2-3}{8u^2-6}\right) - \log \left| \frac{p(z_0)}{p(0)} \right|\right)
\\&\ge \frac{1}{\psi(u)}\left(  \frac{n}2 \log\left( \frac{8u^2-3}{8u^2-6} \right) - \log \frac{\|p\|_\infty}{|p(0)|}\right),
\end{align*}
which, when combining with \eqref{upperpart} yields
\begin{equation*}
8 \frac{d^5}{h^4} M \ge  u^3 k \ge \frac{u^3}{ \psi(u)}  ~
\left(\frac{n}2 \log\left( \frac{8u^2-3}{8u^2-6} \right) -\log \frac{\|p\|_\infty}{|p(0)|}\right).
\end{equation*}
We have already seen in Lemma \ref{l:Hlogp} why it suffices to
restrict to points of the set $\HH$. So from here on we will
consider only points $\zeta \in \HH$, for which points we may
invoke \eqref{eq:pnormperponH} to get
\begin{equation*}
8 \frac{d^5}{h^4} M \ge \frac{u^3}{ \psi(u)}  ~
\left(\frac{n}2 \log\left( \frac{8u^2-3}{8u^2-6} \right) -\log(16\pi)-2 \log n  \right)
\quad \left( ~{\rm for}~\ze \in \HH \right).
\end{equation*}

Introducing another parameter $v:=n/u^4=(h/d)^4 n$ and collecting everything from the above, a little rearrangement leads to
\begin{align*}
& 8 \frac{d^5}{h^4}~M ~\frac{1}{n} \ge
\frac{u^3 \log\left( \frac{8u^2-3}{8u^2-6} \right) - \dfrac2{v u}\left\{\log(16\pi) +2 \log v + 8\log u\right\}}
{\log\left(\frac{8u^2-3}{8u^2-6} \right)+ 2\log(16u)}
\end{align*}
It is clear that for $v\ge e$ this expression is an increasing function of
$v$, therefore we can as well write in the minimal possible value $v\ge v_0=32$ to get
$$
8 \frac{d^5}{h^4}~\frac{M}{n}  \ge
\frac{u^3 \log\left( \frac{8u^2-3}{8u^2-6} \right) - \frac{2}{32u}\left\{\log(\pi) +14 \log 2 + 8\log u\right\}}
{\log\left(\frac{8u^2-3}{8u^2-6} \right)+ 2\log(16u)}
\quad \left( u:=\frac{d}{h} \ge 1 \right).
$$
Here simultaneously dividing the numerator and the denominator by $u$ yields 
\begin{align}\label{eq:Mpernest}
8 \frac{d^5}{h^4}~\frac{M}{n}  & \ge
\frac{u^2 \log\left( \frac{8u^2-3}{8u^2-6} \right) - \frac{1}{16u^2}\left\{\log\pi+ 14\log 2 + 8\log u\right\}}
{\frac{1}{u} \log\left(\frac{8u^2-3}{8u^2-6}\right) + 2 \frac{\log(16u)}{u}}
\quad \left( u:=\frac{d}{h} \ge 1 \right).
\end{align}
As $1/u$, $\log\left(\frac{8u^2-3}{8u^2-6}\right)$ and $2\frac{\log(16u)}{u} =32 \frac{\log(16u)}{16u} $ are all decreasing for $u\ge 1$, the denominator is a decreasing function and its maximal value is at the point $u=1$. So,
$$
\frac 1u\log\left(\frac{8u^2-3}{8u^2-6} \right)+ 2\frac{\log(16u)}{u} \le \log(5/2)+2\log(16) = \log(640) .
$$
From this and \eqref{eq:Mpernest} a computation provides
\begin{equation}\label{e414}
\begin{aligned}
 \frac{d^5}{h^4}~\frac{M}{n}  & \ge
\frac{1}{64 \log 640} \left(8u^2 \log\left( \frac{8u^2-3}{8u^2-6} \right) - \frac{\log(4\pi) + 4\log(8u^2)}{2u^2} \right)
\\& = \frac{1}{413.5339\dots} \left( 8u^2 \log\left( \frac{1-\frac{3}{8u^2}}{1-\frac{6}{8u^2}} \right) - \frac{4\log(4\pi)}{8u^2} - 16\frac{\log(8u^2)}{8u^2} \right)
\\& \ge \frac{1}{414} f(t)
\end{aligned}
\end{equation}
with $\displaystyle f(t):= \frac{1}{t}
\log\left( \frac{1-3t}{1-6t} \right) - 4\log(4\pi)~t + 16t\log t$
and $\displaystyle t:=\frac{1}{8u^2} \in (0,1/8]$.
It is easy to see that $f(t)$  is a convex function. Indeed, $t\log t$ is convex (with second derivative $1/t>0$),
the linear term is of course convex, and the first part can be developed into a totally positive Taylor-Maclaurin series:
$
\frac{1}{t} \log\left( \frac{1-3t}{1-6t} \right) = \sum_{k=1}^\infty \frac{6^k-3^k}{k} t^{k-1} $.

Numerical evidence shows that $f(t)$ attains its minimal value somewhere around $0.0786\dots$ , and it stays above 0.7 all over $(0,1/8]$. To establish a sufficiently good lower estimation of the function all over the interval $(0,1/8]$, we will use convexity simply in the form of a supporting line argument: with any fixed value $\tau$ in $(0,1/8]$ the tangent of $f$ at $(\tau,f(\tau))$ is a supporting line (from below) to $f$, i.e. $f(t) \ge L(t):=L(\tau;t):= f(\tau) + f'(\tau) (t-\tau)$.

A computation furnishes
$$
f'(t)=-\frac{1}{t^2} \log\left( \frac{1-3t}{1-6t} \right) + \frac{1}{t}
\left(\frac{-3}{1-3t} + \frac{6}{1-6t} \right) - 4\log(4\pi) + 16 + 16\log t.
$$
So now let us take $\tau:=0.078628$, say.
Then $f(\tau)\approx 0.700037\ldots > 0.70003$, and another numerical computation furnishes
$f'(\tau) \approx -0.000321\ldots > - 0.0004$. Since $f'(\tau)<0$, we find
$f(t) \ge \min_{(0,1/8]} L =L(0.125) > 0.70003 - 0.0004 \cdot (0.125-0.07) = 0.700008 > 0.7$.


Substituting this estimate in~\eqref{e414}, we conclude
$$
M \ge \frac{0.7}{414}\frac{h^4}{d^5}n>\frac{1}{600}\frac{h^4}{d^5}n.
$$

\end{proof}

\subsection{Conclusion of the proof for fixed positive depth}\label{ss:posdepthpr}


\begin{proof}[Proof of Theorem \ref{th:posdepth}]
We will use for all points $\ze\in \HH$ the estimate
\eqref{eq:ptwisenh} complemented by the lower estimation $h\ge
h_K$.
$$
\|p'\|_q^q \ge \int_\HH |p'|^q \ge \left(\frac{1}{1500} \frac{h_K^4}{d^5}\right)^q \int_\HH |p|^q
\ge \left(\frac{1}{1500} \frac{h_K^4}{d^5}\right)^q \frac12 \|p\|_q^q.
$$
Taking $q^{\rm th}$ root and estimating $2^{1/q}$ simply by 2
yields Theorem \ref{th:posdepth}.
\end{proof}

\section{Upper estimation of the oscillation order of convex domains}\label{sec:osciupper}

Given the results for maximum norm and the above results of Theorem \ref{th:Lqcircular} and Theorem \ref{th:posdepth}, it is in order to clarify if the linear growth with $n$ is indeed the maximal possible order of oscillation in $L^q$ norms. That is settled by the next result.

\begin{theorem} Let $K\Subset \CC$ be any compact, connected--not necessarily convex--domain, bounded by a finite or countable number of closed, rectifiable Jordan curves $\Gamma_j$ ($j=1,\dots)$ with finite total arc length $\sum_{j} |\Gamma_j| = L <\infty$. Then for any $n\in \NN$ there exists some polynomial $p\in \PK$ with $\|p'\|_{L^q(\DK)} \le C(K) n \|p\|_{L^q(\DK)}$.
\end{theorem}

Note that the rectifiable assumption is necessary to have finite $L^q$ norms, for otherwise most polynomials have infinite $L^q$ norms on the boundary. However, apart from this assumption, the domain $K$ is quite general, including nonconvex, multiply connected domains. For disconnected domains, the analysis may be done separately for connected components, and for compact sets without an interior even a lower order of oscillation is possible, as it has been shown at least for the interval $\II$. Therefore, we may be satisfied with the degree of generality of the above formulated assertion.

\begin{proof} We will provide a simple example. Let $J\subset K$ be any diameter, and chose an endpoint of the diameter $J$. Without loss of generality we may assume that this endpoint is just the origin $0$, and we can as well assume that $J=[0,d]$. Then our polynomial will simply be $p(z):=z^n$.

At points, where $|z|\le d/2$, we have $|p(z)| \le (d/2)^n = 2^{-n} d^n = 2^{-n} \|p\|_K$, and $|p'(z)|= n |z^{n-1}| \le n2^{-(n-1)} d^{n-1}$. On the other hand, for points in the ring domain $R:=\{ z\in \CC~:~ d/2 \le |z| \le d\}$ and belonging to $K$ we have $\left|\frac{p'}{p}(z)\right|=\left|{n}/{z}\right|<2n/d$. So  we can write
\begin{align}\label{znprimeandzn}
\|p'\|_q^q &:=\int_\Gamma |p'|^q \le (2n 2^{-n} d^{n-1} )^q L + \int_{\Gamma \cap R} \left( \frac{2n}{d} |p| \right)^q \notag
\\ & \le \left( \frac{2n}{d} \right)^q \left\{ \left( 2^{-n} \frac{\|p\|_\infty}{\|p\|_q}\right)^q  L +1 \right\}\|p\|_q^q.
\end{align}
By construction, the point $d$ on the other end of the diameter $J$ sits in $\DK$, and belongs to some of the boundary curves $\Gamma_j$, which boundary curve must have some positive length $|\Gamma_j|=\ell>0$, say. So parameterizing by arc length and starting the parametrization at the point $d$, we can write $\gamma_j:[0,\ell]
\to \Gamma_j$ with $\gamma_j(0)=d=\gamma_j(\ell)$, and obviously for any parameter value $0\le t \le d$
$|\gamma_j(t)| \ge d - t$, since the arc $\gamma_j\big|_{[0,t]}$ cannot go farther from the left endpoint at $d$ then its arc length $t$. It follows that at $z=\gamma_j(t)$ it holds $|p(z)|\ge (d - t)^n$ until $0\le t \le \lambda:=\min(d,\ell)$, and so we have
\begin{align*}
\|p\|_q^q& \ge \int_{\Gamma_j} |p(z)|^q |d(z)| \ge \int_0^{\lambda} (d-t)^{nq} dt =
\frac{1}{nq+1} \left[ d^{nq+1} - (d-\lambda)^{nq+1}\right]
\\
& \ge \frac{d\left[ d^{nq} - (d-\lambda)^{nq}\right]}{nq+1}
\end{align*}
and
$$
\left( \frac{\|p\|_\infty}{\|p\|_q}\right)^{q} \le \frac{nq+1}{d}  \left( \frac{1}{1-(1-\lambda/d)^{nq}} \right) < \frac{nq+1}{d}  \left( \frac{1}{1-(1-\lambda/d)} \right) = \frac{nq+1}{\lambda}
$$
Finally applying this in \eqref{znprimeandzn} leads to
$$
\|p'\|_q^q \le \left( \frac{2n}{d}\right)^{q} \left\{\frac{nq+1}{\lambda 2^{nq}} L +1 \right\} \|p\|_q^q.
$$
Since $q\ge 1$ and $nq+1 \ge 2$, it suffices to observe that $x/2^x$ decreases for $x\ge 2$ and thus $(nq+1)2^{-nq} \le 2 \max_{x\ge 2} x2^{-x} = 1$. We finally obtain
\begin{equation}\label{explicitconstforupper}
\|p'\|_q \le \frac{2}{d}\frac{L+\min(d,\ell)}{\min(d,\ell)} n \|p\|_q.
\end{equation}
As here the constant depends only on the domain $K$, we conclude that $\|p'\|_q \le C(K) n \|p\|_q$ holds for the chosen polynomial $p\in \PK$, whence the assertion.
\end{proof}

\begin{remark} In case of a convex domain $K$, the parameters occurring here in $C(K)$ have a simpler meaning. First, the boundary $\DK$ is connected (consists of only one convex curve), and thus $\ell=L$ and $\min(d,\ell)=\min(d,L)=d$. Second, as we have used several times, \emph{for convex curves} the estimate $L \le 2\pi d$ holds true, always, whence \eqref{explicitconstforupper} simplifies to $\|p'\|_q \le \dfrac{4\pi+2}{d} n \|p\|_q < \dfrac{15}{d} n \|p\|_q $, say.
\end{remark}

\begin{remark} In \cite{Rev2} a more precise value of the constant $C(K)$ has also been obtained for the case of the infinity norm. As for $L^q$-norm neither the order (in general), nor the more exact constants are known, it seemed to be well ahead of time to bother with sharper values of the constant $C(K)$ here. Nevertheless, our feeling is that the slightly more involved construction of \cite{Rev2} would indeed provide a better constant, which may be sharp, apart from an absolute constant factor, like in case $q=\infty$. Here we do not pursue this issue any more.
\end{remark}

\section{Concluding remarks}\label{sec:conclusion}

Above we have seen, that like in case of the maximum norm, also
for the $L^q(\DK)$ norm any compact convex domain $K$ admits
polynomials $p\in \PK$ with oscillation not exceeding $O(n)$.
On the other hand we have shown for some classes of convex
domains that the order of oscillation indeed reaches $c_K n$.

A natural question--quite resembling to the question posed by Er\H od 77 years ago in case of the maximum norm--is to identify those domains which indeed admit order $n$ oscillation even in $L^q(\DK)$ norm.

It has been clarified that, like in case of the maximum norm,
also for $L^q$ norms the interval $\II$ behaves differently:
there the order of oscillation may be as low as
$\sqrt{n}$. Therefore, it is certainly necessary that
\emph{some} conditions are assumed for an order $n$
oscillation. The question is if apart from having a nonempty
interior, is there need for any additional assumption? We
think that probably not.

\begin{conjecture} For all compact convex domains $K\Subset \CC$ there exist $c_K>0$ such that for any $p\in\PK$ we have $\|p'\|_{L^q(\DK)} \ge c_K n \|p\|_{L^q(\DK)}$.
\end{conjecture}

We are not really close to this conjecture. Let us point out that a general estimate--however weak--is still missing in the full generality of all convex compact domains. Apart from the cases discussed here, we mentioned that the case of the interval $\II$ is clarified--there the oscillation being of order $\sqrt{n}$. So clearly also here there is a difference between various compact convex sets. However, we do not really know if $\II$ is indeed to be ``the worst", i.e. of lowest possible oscillation, as for general domains really nothing--not even some mere $\log n$ e.g.--has been proved to date.

Therefore, posing a more modest goal, we would be interested as well \emph{in any estimate} working for \emph{general compact convex domains}, without further assumptions on the geometrical features of it.

\section*{Acknowledgments}
This work was supported by the Russian Foundation for Basic Research (Project No. 15-01-02705)
 and by the Program for State Support of Leading Universities of the Russian Federation (Agreement No. 02.A03.21.0006 of August 27, 2013)
and by Hungarian Science Foundation Grant \#'s K-100461,  NK-104183, K-109789.


\end{document}